%% file: main.tex
\documentclass[reqno]{amsart}

\input{preamble.tex}

\title[FAST LDM WITH ENERGY-DISTANCE]{
    \Large{Fast Large Deformation Matching with the Energy Distance Kernel.}}
\author{Siwan Boufadene$^\dagger$}
\author{Jean Feydy$^\diamond$}
\author{François-Xavier Vialard$^\dagger$}

\begin{document}

\vspace{-8mm}
\begin{abstract}
    We propose an efficient framework for point cloud and measure registration using bi-Lipschitz homeomorphisms, achieving $O(n\log n)$ complexity, where $n$ is the number of points. By leveraging the Energy-Distance (ED) kernel, which can be approximated by its sliced one-dimensional projections, each computable in $O(n\log n)$, our method avoids hyperparameter tuning and enables efficient large-scale optimization. 
    The main issue to be solved is the lack of regularity of the ED kernel. To this goal, we introduce two models that regularize the deformations and retain a low computational footprint. The first model relies on TV regularization, while the second model avoids the non-smooth TV regularization at the cost of restricting its use to the space of measures, or cloud of points.  Last, we demonstrate the numerical robustness and scalability of our models on synthetic and real data. 
\end{abstract}

\maketitle

\vspace{1mm}

\vspace{2mm}

\vspace*{\fill}

Email addresses: \href{mailto:siwan.boufadene@live.fr}{siwan.boufadene@live.fr}, 
\href{mailto:jean.feydy@inria.fr}
{jean.feydy@inria.fr},
\href{mailto:francois-xavier.vialard@u-pem.fr}{francois-xavier.vialard@univ-eiffel.fr}

\hypersetup{linkcolor=black}

\setcounter{tocdepth}{3}
{
    \hypersetup{linkcolor=black}
    \tableofcontents
}

\allowdisplaybreaks
\input{intro}

\input{intro_lddmm}
\input{rkhs_cpd_case_lddmm}

\input{ed_ker_registration}
\input{Sliced_ED}
\input{experiments}

\newpage
\clearpage
\markboth{}{}
\section*{Acknowledgements}
The authors want to warmly thank Stefan Sommer and his co-authors for sharing their work around conditionally positive kernels and stimulating discussions on this topic.

\vspace{0.2cm}

\noindent
\textbf{Funding: }This work was supported by the Bézout Labex (New Monge Problems), funded by ANR, reference ANR-10-LABX-58.

\vspace{0.2cm}

\printbibliography
\clearpage
\markboth{APPENDIX}{APPENDIX}
\newpage
\appendix
\clearpage
\markboth{APPENDIX}{APPENDIX}
\input{appendix}

\end{document}

%% file: preamble.tex
\usepackage[bottom=3cm, top=3cm, left=3cm, right=3cm]{geometry}
\usepackage{amsmath}
\usepackage{amsthm}
\usepackage{amsfonts}
\usepackage{amssymb}
\usepackage{etoolbox}
\usepackage{array}
\usepackage{algorithm}
\usepackage{algpseudocode}
\usepackage{mathtools}

\usepackage{cases}
\usepackage{empheq}
\usepackage[shortlabels]{enumitem}
\setlist[itemize]{noitemsep,nolistsep}
\setlist[enumerate]{noitemsep,nolistsep}
\usepackage{tabularx}
\usepackage{stmaryrd}

\usepackage[dvipsnames]{xcolor}
\usepackage{graphicx}
\input{colors}
\usepackage{tikz}
\usepackage{tikz-cd}

\usepackage{hyperref}
\hypersetup{
  breaklinks=true,
  colorlinks=true,
  linkcolor=linkcolor,
  urlcolor=urlcolor,
  citecolor=citecolor,
  }
  
\usepackage[
    style=alphabetic,       
    backend=bibtex,          
    maxcitenames=2,         
    maxbibnames=10,         
    minbibnames=3,          
    giveninits=true,        
    url=false,              
    doi=false,              
    isbn=false,             
    backref=false,          
    hyperref=true,          
    alldates=short,         
    sorting=nyt             
]{biblatex}
\usepackage[noabbrev,capitalize,nameinlink]{cleveref}

\crefname{equation}{}{}
\Crefname{equation}{Eq.}{}

\usepackage{thmtools}
\usepackage{framed}

\input{commands}
\input{project_specific}

\makeatletter
\renewcommand{\tocsection}[3]{%
  \indentlabel{\@ifnotempty{#2}{\bfseries\ignorespaces#1 #2\quad}}\bfseries#3}
\renewcommand{\tocsubsection}[3]{%
  \indentlabel{\@ifnotempty{#2}{\ignorespaces#1 #2\quad}}#3}
\def\l@subsection{\@tocline{2}{0pt}{2.5pc}{5pc}{}}
\def\l@subsubsection{\@tocline{3}{0pt}{3.5pc}{7pc}{}}
\makeatother

\allowdisplaybreaks

\newcounter{counter}
\numberwithin{counter}{section}
\newtheorem{theorem}[counter]{Theorem}
\newtheorem{lemma}[counter]{Lemma}
\newtheorem{proposition}[counter]{Proposition}
\newtheorem{corollary}[counter]{Corollary}
\theoremstyle{definition}
\newtheorem{remark}[counter]{Remark}
\newtheorem{example}[counter]{Example}
\newtheorem{definition}[counter]{Definition}

\numberwithin{equation}{section}
\numberwithin{equation}{section}

%% file: colors.tex
\definecolor{tabblue}{rgb}{0.12156862745098039, 0.4666666666666667, 0.7058823529411765}
\definecolor{taborange}{rgb}{1.0, 0.4980392156862745, 0.054901960784313725}
\definecolor{tabgreen}{rgb}{0.17254901960784313, 0.6274509803921569, 0.17254901960784313}
\definecolor{tabred}{rgb}{0.8392156862745098, 0.15294117647058825, 0.1568627450980392}
\definecolor{tabpurple}{rgb}{0.5803921568627451, 0.403921568627451, 0.7411764705882353}

\colorlet{citecolor}{tabgreen}
\colorlet{linkcolor}{tabblue}
\colorlet{urlcolor}{taborange!75!black}

%% file: project_specific.tex
\usepackage[colorinlistoftodos,prependcaption,textsize=tiny,textwidth=2.6cm]{todonotes}

\makeatletter
\@mparswitchfalse%
\makeatother
\normalmarginpar
\makeatletter
\providecommand\@dotsep{5}
\renewcommand{\listoftodos}[1][\@todonotes@todolistname]{%
  \@starttoc{tdo}{#1}}
\makeatother



%% file: intro.tex
\section{Introduction}
\label{sec:introduction}

The analysis of shape variations is a well-studied domain, with several applications in biomedical imaging \cite{rueckert1999,DAVATZIKOS1997207}. One of the standard tasks in this field is registering shapes, which can be of different natures, images, segmented surfaces, line bundles, etc... Several methods have been proposed over the last two decades, which can be divided into two main categories: the first category is optimization-based methods, in the LDM framework for example \cite{trouve_miller_younes_2005,vialard2012}, and the second category is learning-based methods \cite{Dalca_2019,zhu2021testtimetrainingdeformablemultiscale}. Typically, the second method requires a large database. Despite the name, recent foundation/universal models \cite{tian2024unigradiconfoundationmodelmedical} may not generalize well for datasets that are extreme outliers to the set of training data classes. Fine-tuning of these models is recommended, particularly under supervision, which optimization-based methods can provide. However, these methods may suffer from long computational times, depending on the dimensionality of the input. 

In this paper, we are interested in adapting regular and invertible shape registration methods in the so-called large deformation diffeomorphic metric mapping (LDDMM) framework \cite{tro98,trouve_miller_younes_2005,glaun2005} to the Energy-Distance kernel \cite{Sejdinovic_2013}. These deformations, which can be applied to images, clouds of points, or measures, are parameterized by time-dependent vector fields $v_t \in L^2([0,1],V)$, where $V$ is a (sufficiently smooth) reproducing kernel Hilbert space of vector fields defined on $\R^d$, such as Sobolev spaces of high-order. The corresponding flow $\psi_t$ is obtained by solving $\partial_t \psi(t,x) = v(t,\psi(t,x))$ with the initial condition $\psi(0,x) = x \in \R^d$.
Let $A$ be a template object and $B$ a target. The registration problem minimizes the energy functional:
\begin{equation}\label{EqVariationalProblemLDDMM}
    \int_0^1 \frac 12 \| v_t \|_{V}^2 dt + \operatorname{Sim}( \psi(1)\cdot A,B)\,,
\end{equation}
where $\operatorname{Sim}$ is a similarity measure and $\psi \cdot A$
 represents a natural transformation of the object $A$ under the action of the diffeomorphism $\psi(1)$. For instance, this action is simply its image by $\psi(1)$ for point clouds. For measures, the action is the push-forward action of the map $\psi(1)$. By the Representer Theorem, the optimal $v_t$ for a collection of points $(x_i)_{i = 1,\ldots,n}$ is written as $K \star \sum_{i = 1}^n p_i(t) \delta_{x_i(t)}$ where $K$ is the kernel of the RKHS and $p_i \in \R^d$. While this reduces the problem to finite dimensions, the regularization term $\sum_{i,j} \langle p_i ,k(x_i,x_j)p_j \rangle$ retains a quadratic computational cost, typical of general kernel interaction functionals, for which trade-offs between memory and computations are sometimes necessary \cite{Keops,NEURIPS2020_a6292668}. To overcome this computational bottleneck, various acceleration techniques have been proposed in these contexts, such as the fast multipole method \cite{greengard1987fast,ying2004kernel} or other proximity-based algorithms such as k-d tree (\cite{LI201588} for an application in the Iterative Closest Point method). While these methods achieve efficiency by approximating the kernel, they often compromise the positivity of the kernel, a property essential for the well-posedness of the variational problem \eqref{EqVariationalProblemLDDMM}. 
 As a result, such approximations limit the applicability of these methods within the LDDMM framework. Another approximation technique, called particle mesh method \cite{Cotter_2008}, consists of using a mesh to encode the quantities of interest and exchange information between the particles $(x_i)_{i = 1,\ldots,n}$ and the grid. On the mesh, FFT can be used for a separable kernel to reduce the computational burden in $O(n\log(n))$, where $n$ is the total grid size, by leveraging a regular grid structure. However, due to computational constraints, grid-based methods become very impractical for ambient dimensions greater than 4.

\vspace{0.1cm}

\noindent{\textbf{Main contribution:} }
The first contribution of this article is to provide a simple computational framework for point clouds and measures registration in the LDMM framework with the Energy-Distance (ED) kernel. We prove that not only the gradient of the ED loss \cite{hertrich2024generativeslicedmmdflows}, but also the convolution of any finite measure with the ED kernel can be approximated by a finite sum of its one-dimensional projections, with each projection being computable with an $O(n\log(n))$ time complexity. Using an Energy-Distance based loss function, this yields an efficient registration algorithm, which can be scaled to very large datasets. In addition, the ED kernel does not require any additional hyperparameter tuning, in sharp contrast with a mixture of Gaussian kernels. 
Using the ED kernel, we propose a registration algorithm that generates bi-Lipschitz homeomorphisms, which is a class of regular and invertible deformations strictly larger than diffeomorphisms. In that sense, our registration model is more expressive than standard LDDMM.
\vspace{0.1cm}

\noindent{\textbf{Main challenges:} } 
We propose to adapt the use of the energy distance kernel in the Euclidean space, defined by $K(x,y) = -| x -y|$. There are indeed several challenges that hinder the direct application of LDDMM theory with this kernel, which must be addressed to ensure theoretical well-posedness and practical feasibility.
First, while the ED kernel is not positive definite but only conditionally positive, this can be resolved by constraining the so-called momentum measure $\sum_{i = 1}^n p_i \delta_{x_i}$ to have zero-mean. Although this constraint reduces expressivity, the Energy-Distance registration framework naturally handles the addition of a translation term, maintaining universal approximation properties.

The second issue is the smoothness of the corresponding Reproducing Kernel space, which is not contained in $C^1$ functions. Typically, the Energy-Distance kernel only generates Hölder-continuous velocity fields. Since such velocity fields fail to correspond to a well-defined unique flow, additional regularization is needed in theory. We propose two parameterizations, retaining the lightweight cost of the kernel, which guarantee that the resulting deformation is a bi-Lipschitz homeomorphism. While being theoretically more expressive than LDDMM deformations, which are smooth and invertible, our formulation maintains key theoretical properties such as completeness of the group and the existence of minimizers to the variational problem. 

\vspace{0.1cm}

\noindent{\textbf{Structure of the paper: } } 
After a brief overview of the LDDMM framework, we present our main contributions. 
In section \ref{SecInterpolation}, we study in detail the finite interpolation problem using the ED kernel and general conditionally positive kernels. We identify the corresponding velocity fields space as a Beppo-Levi space, providing regularity results for the corresponding function space. In Section \ref{SecTheNeedForRegularization}, we study the registration problem for point clouds using the ED kernel. While proving the existence of solutions, we demonstrate potential non-injectivity, motivating our two regularization approaches. In the first model, we use a Total Variation (TV) regularization of the momentum-measure. A Representer Theorem adapted to TV regularization maintains the $O(n\log n)$ time-complexity. Our second model only deals with measure matching, where we apply a $L^2 $ regularization of the momentum w.r.t. the current measure. This formulation is easier to optimize numerically and applies to point clouds and discrete measures. In Section \ref{SecSlicedED}, we present how the ED kernel computations can be sliced with $O(n \log n)$ footprint, providing statistical error bounds for the approximations. In Section \ref{SecExperiments}, we detail an $O(n\log n)$ algorithm for registration, and evaluate them on synthetic and real data. The ED kernel and its sliced version frequently outperform standard kernels w.r.t. the robustness of the optimization. Moreover, it requires \textbf{no hyperparameter} tuning, offering significant advantages in various scenarios.

\vspace{0.1cm}

\noindent{\textbf{Perspectives: } } 
Our light computational framework for LDDMM opens the way to its use in two important applications: processing small but high-precision datasets, and handling high-dimensional biomedical data \cite{SingleCellsOT}, possibly across different modalities. 
The energy functional consists of two components: the energy, which can be computed in 
$O(n\log n)$, and the data discrepancy, which is particularly crucial in practice. Recent advancements have demonstrated the effectiveness of incorporating higher-order information, such as through currents or varifolds \cite{Charon_2013}. An open question remains to develop higher-order similarity measures between point clouds that can also be computed in $O(n\log n)$.

From a theoretical perspective, our experiments suggest the further study of a Hamiltonian formulation, which appears well-suited for this framework.

%% file: intro_lddmm.tex
\subsection{Introduction to Large Deformation by Diffeomorphisms}

This section provides a brief overview of the mathematical framework for large deformations by diffeomorphisms (see for instance \cite{glaun2005}). 

\subsubsection{Deformation group}
The idea is to model diffeomorphisms, i.e. smooth invertible maps, using displacement fields $v_t \in V$, where $V$ is a space of sufficiently smooth vector fields. In such a case, it is called admissible.

\begin{definition}[Admissible Hilbert spaces]
    A Hilbert space $V$ is said to be \textit{admissible} if
    \begin{equation}
        (V,\|\cdot \|_{V}) \hookrightarrow (\mathcal{C}_0^1(\R^d,\R^d),\| \cdot \|_{1,\infty}) \,,
    \end{equation}
    where $\mathcal{C}_0^1(\R^d,\R^d)$ is the space of $\mathcal{C}^1$ fields that tend to zero at infinity, along with their partial derivatives. 
\end{definition}

\begin{example}
    The Sobolev space $H^s(\R^d,\R^d)$, where $s > 0$, is defined as
    \begin{equation}
        H^s(\R^d,\R^d) = \left\{ v \in L^2(\R^d,\R^d) \, \bigg|\, \|v\|_{s}^2 \coloneqq \int (1+|\xi|^2)^s |\hat{v}(\xi)|^2 d\xi < \infty \right\} \,,
    \end{equation}
    A classical result is that if $s > \frac{d}{2} +1$, by the Sobolev embedding theorem: 
    \begin{equation}
        (H^s(\R^d,\R^d),\|\cdot \|_{\mathcal{H}^s}) \hookrightarrow (\mathcal{C}_0^1(\R^d,\R^d),\| \cdot \|_{1,\infty}) \,.
    \end{equation}
\end{example}
We will now denote $V$ an admissible space for the rest of the section, and define the time-dependent spaces
\begin{equation}
    L^2([0,1],V) = \left\{ v:(t,x) \in [0,1] \times \R^d \mapsto v_t(x) \in V \, \bigg|\, \int_0^1 \|v_t\|_V^2dt < \infty \right\}\,,
\end{equation}
and 
\begin{equation}
    L^1([0,1],V) = \left\{ v:(t,x) \in [0,1] \times \R^d \mapsto v_t(x) \in V \, \bigg|\, \int_0^1 \|v_t\|_Vdt < \infty \right\} \subset L^2([0,1],V)\,.
\end{equation}
The admissibility condition is enough to prove the following theorem:
\begin{theorem}[see \cite{younes2010shapes}]\label{th:v_admi_diffeo}
    Let $v \in L_V^1$. Then, for all $x \in \R^d$, there exists a unique continuous solution $t \mapsto \psi_t^v(x)$ that verifies the Lagrangian flow equation
    \begin{equation}
        \psi_t^v(x) = x + \int_0^t v_s \circ \psi_s^v(x)ds \,.
    \end{equation}
    Moreover, for all $t\geq 0$, the application $\psi_t^v$ is a $\mathcal{C}^1$ diffeomorphism of $\R^d$. Finally, the applications $(t,x) \mapsto \psi_t^v(x)$ are uniformly continuous, in a uniform way regarding $v$ on every bounded set in $L_V^2$.
\end{theorem}
The deformation space of interest is then defined as
\begin{equation}
    \mathcal{A}_V \coloneqq \{ \psi_1^v \,|\, v \in L_V^2 \} \,,
\end{equation}
which verifies the following property:
\begin{proposition}[(see \cite{tro98})]
    The function space $\mathcal{A}_V$ is a diffeomorphism group, which is complete for the metric defined by
    \begin{equation}
        d(\id,\psi) \coloneqq \inf \left\{ \|v\|_{L_V^2} \,|\, v \in L^2_V, \psi = \psi_1^v \right\}\,,
    \end{equation}
    and $d$ is extended to arbitrary diffeomorphisms $\phi,\psi \in \mathcal{A}_V$ by right-invariance $d(\phi,\psi) \coloneqq d(\id,\psi \circ \phi^{-1})$.
\end{proposition}
The deformation space $\mathcal{A}_V$ is a geodesic space:
\begin{proposition}[Existence of constant speed geodesics]
    Let $\phi,\psi \in \mathcal{A}_V$. Then there exists $v \in L_V^2$ such that $\psi_1^v = \psi \circ \phi^{-1}$ and $d(\phi,\psi) = \|v\|_{L_V^2} = \|v\|_{L_V^1}$.
\end{proposition}

\subsubsection{Kernel methods}
In applications, the admissible space $V$ is a Reproducible Kernel Hilbert Space associated with a given positive definite kernel $K$, as provided by Aronzajn's Theorem \ref{th:Aronszajn}. Very often, kernels of the form $K(x,y) = h(|x-y|)\id$ are used, where $h$ is a scalar function.
\begin{example}
    \begin{itemize}
        \item \textbf{Gaussian kernel.} Choosing the Gaussian kernel defined, for $\sigma >0$, through: $$k_{\sigma}(x,y) = \exp\left( -\frac{|x-y|^2}{2\sigma^2} \right)\,,$$ we get the Hilbert space $\mathcal{G}_{\sigma}$, which is embedded in $\mathcal{C}^k(\R^d,\R^d)$ for all $k\geq 0$. Remark that $k_{\sigma'} \in \mathcal{G}_{\sigma}$ if and only if $\sigma ' > \sigma / \sqrt{2}$, so that the velocity fields spaces $\mathcal{G}_{\sigma}$ are all distinct.
        \item \textbf{Sobolev kernel.} In general, describing the generated deformation group $\mathcal{A}_{V}$ for a given admissible space $V$ is a tricky operation. However, in the case of the Sobolev kernel, whose RKHS is the Sobolev space $H^s$, the following result was proven in \cite{bruveris2016completenessgroupsdiffeomorphisms}:
        $$\mathcal{A}_{H^s} = (\mathcal{D}_s(\R ^ d))_0\,,$$
        where $(\mathcal{D}_s(\R ^ d))_0$ is the connected component of the identity in the group of Sobolev diffeomorphisms.
    \end{itemize}
\end{example}

\subsubsection{General matching problem}

In this paragraph, $V$ is an admissible space.
\begin{definition}[General matching problem]\label{def:match_gen}
    Let $A : \mathcal{A}_{V} \to \R^+$ be a discrepancy function and $\lambda >0$. The corresponding minimal matching problem is defined as
    \begin{equation}
        \underset{\psi \in \mathcal{A}_{V}}{\min}  A(\psi)+\lambda d_V(\id,\psi)  \,,
    \end{equation}
    or in an equivalent way
    \begin{equation}
        \underset{v \in L_V^2}{\min}  A(\psi_1^v) + \lambda\int_0^1\|v_t\|_V^2  \,.
    \end{equation}
\end{definition}
The following result can be written:
\begin{theorem}[Existence of a minimizer]\label{th:gen_match_exis_min}
    Let $A$ be defined as in \ref{def:match_gen}. Let us suppose that $v \mapsto A(\psi_1^v)$ is weakly continuous from $L_2^V$ to $\R$. Then there always exists a solution to the matching problem from Definition \ref{def:match_gen}.
\end{theorem}

\subsubsection{Measure matching using a positive definite kernel.}

The previous general matching problem can be instantiated in the context of measure matching. Let $\mu_0$ be a probability measure. If $\psi_t^v$ is the $\mathcal{C^1}$ diffeomorphism from Theorem \ref{th:v_admi_diffeo}, then the time-dependent measure $\mu_t \coloneqq \psi_t^v \# \mu_0$ is a solution in the sense of distributions to the continuity equation:
\begin{equation}
    \partial_t \mu_t +  \nabla \cdot (v_t \mu_t) = 0 \,,
\end{equation}
for $t \in (0,1)$ \cite[8.1]{ambrosio2005gradient}. By a direct application of Theorem \ref{th:gen_match_exis_min}, there exist solutions to the following matching problem:
\begin{proposition}
    Let $\mu_0, \nu \in \mathcal{P}(\Omega)$ and $\mathcal{L}: \mathcal{P}(\Omega) \times \mathcal{P}(\Omega) \mapsto \R^+$ be a lower semi-continuous loss, i.e. a function such that $\mathcal{L}(\alpha,\beta) = 0$ iif $\alpha = \beta$ in $\mathcal{P}(\Omega)$. Then the application $v \in L_2^V \mapsto \mathcal{L}(\psi_1^v \# \mu_0,\nu)$ is weakly continuous, and for $\lambda > 0$ the matching problem
    \begin{equation}
        \underset{v \in L_2^V}{\min}\, E(v) \coloneqq \mathcal{L}(\psi_1^v \# \mu_0,\mu_1) + \int_0^1 \|v_t\|_V^2 dt \,
    \end{equation}
    has a solution in $L_V^2$. 
\end{proposition}

\begin{example}\label{ex:loss_examples}
    There are multiple possible choices for $\mathcal{L}$. Some of them, that are specific to this measure context, are:
    \begin{itemize}
    \item Maximum Mean Discrepancy losses: If $K$ is a kernel, define $\mathcal{L}_K(\mu,\nu) \coloneqq \langle \mu - \nu, K \star (\mu - \nu) \rangle$. One can choose $K$ to be a Gaussian kernel, Laplacian kernel, or the Energy-Distance kernel. For more information see \cite{sriperumbudur2010hilbertspaceembeddingsmetrics}. In general, these discrepancies have a quadratic computational cost in terms of the number of points used to represent the measure $\mu$. However, for the ED kernel, an approximation of it can be reduced to a $O(N\log(N))$ computational cost, which makes it attractive.
    
    \item A natural choice is to choose the optimal transport loss $\mathcal{L}(\mu,\nu) = W_2^2(\mu,\nu)$, where $W_2$ is $L^2$ Wasserstein distance, or its entropic regularization \cite{feydy2017optimaltransportdiffeomorphicregistration,delara2022diffeomorphicregistrationusingsinkhorn}. These two losses are at the very least quadratic if not cubic.
    An advantageous loss function derived from OT ideas is the sliced Wasserstein distance, whose cost can be approximated in $O(n\log(n))$.
\end{itemize}
\end{example}

\subsection{Reduction of the optimization set.}
The optimization set $L^2([0,1],V)$ can be reduced to $L^2([0,1],(\R^d)^n)$ in the case of empirical measure matching with $n$ Dirac masses (or cloud of $n$ points).
To fix the setting, we consider $\mu = \frac 1n \sum_{i = 1}^n \delta_{x_i}$ as the source measure. Given the trajectory of the points $x_i(t) = \varphi_t(x_i)$, the minimizing vector field of the energy functional $\int_0^1 \| v_t\|^2_{V} dt $ admits a sparse representation due to the Representer Theorem for RKHS. Specifically, $v_t$ is written as $v_t(x) = \sum_{i = 1}^n p_i(t) \delta_{x_i(t)}$, where $p_i(t) \in L^2([0,1],(\R^d)^n)$. In terms of these new variables, the energy functionnal reads:
\begin{equation}
    \frac{1}{2} \int_0^1 \langle p_i(t) , K(x_i(t),x_j(t)) p_j(t)\rangle dt + \mathcal{L}_K\left(\frac 1n \sum_{i = 1}^n \delta_{x_i(1)},\nu\right)\,,
\end{equation}
under the constraints $\dot{x}_i(t) = \sum_{j = 1}^n K(x_i(t),x_j(t))p_j(t)$.
Without any further additional structure, the computational cost of this functional is quadratic in $n$. 

The goal of the following sections is to adapt the ED kernel to the LDMM framework, so that efficient $O(n\log(n))$ algorithm can be performed.

%% file: rkhs_cpd_case_lddmm.tex
\section{Interpolation using the Energy-Distance kernel, induced metrics and regularity.}
\label{SecInterpolation}
In this section, we present results on interpolation using a conditionally positive definite kernel. The motivation is as follows: Let $\mu$ be a probability measure on $\R^d$. To transport the measure, we consider the action of a vector field $v$ whose restriction to supp$(\mu)$ is in $L^2(\mu)$, defined via
\begin{equation}
    \mu \cdot v = (\Id + v)_\#\mu \,,
\end{equation}
where $_\#$ is the push-forward operation. Here, $v$ is chosen to belong to a Hilbert space (resp. semi-Hilbert space), associated with a PD (resp. CPD) kernel. This space is constructed as the tensorization of the RKHS (resp. $\mathbb{P}$-RKSHS) with values in $\R$ associated with the kernel. 
In the two following subsections, we present the theory of interpolation theory for the case of scalar values.

\subsection{Interpolation using a CPD kernel.}

In this paragraph, we present the theory of interpolation using a positive definite kernel, and how it extends to conditionally positive definite (CPD) kernels \cite{auffray_pdcpker_2009}.

\subsubsection{Case of a positive definite kernel.}

Let $K : \Omega \times \Omega \rightarrow \R$ be a symmetric positive definite kernel, meaning that for any $(x_1, \dots, x_n) \in \Omega^n$ and $(c_1, \dots, c_n) \in \R^n $, we have
\begin{equation}
    \sum_{i,j = 1}^n c_i c_j K(x_i,x_j) \geq 0 \,,
\end{equation}
with equality if and only if $c_1 = \dots = c_n = 0$. Define the partial function $K_x : y \in \Omega \mapsto K(x,y)$, and let $\mathcal{F}_K$ be the space of linear combinations of functions in the set $\{K_x ; x\in \Omega\}$.
The inner product 
\begin{equation}
    \langle \lambda_i K_{x_i} , \beta_j K_{y_i} \rangle_{\mathcal{F}_K} \coloneqq \sum_{i,j} \lambda_i \beta_j K(x_i,y_j)
\end{equation}
defines a symmetric positive bilinear form on $\mathcal{F}_K$.
Now, an implication of Aronszajn's famous theorem \cite{aronzajn1950} states:
\begin{theorem}[Aronszajn]\label{th:Aronszajn}
    \begin{enumerate}
        \item The bilinear form $\langle \cdot , \cdot \rangle_{\mathcal{F}_K}$ is positive definite.
        \item There exists a unique Hilbert space of functions $\mathcal{H}_K$ such that
        {\begin{enumerate}
            \item $(\mathcal{F}_K,\langle \cdot , \cdot \rangle_{\mathcal{F}_K})$ is a pre-Hilbert subspace of $\mathcal{H}_K$ \,.
            \item (Reproducing property) $\forall f \in \mathcal{H}_K, x \in \Omega, f(x) = \langle f, K_x \rangle_{\mathcal{H}_K} \,.$
        \end{enumerate}}
    \end{enumerate}
    The space $\mathcal{H}_K$ is called the Reproducible Kernel Hilbert Space associated with $K$.
\end{theorem}
In spline interpolation, the goal is to approximate a function $f \in \mathcal{H}_K$ that is known on a finite design set $X = \{x_1, \dots, x_N\} \subset \Omega$ using a finite-dimensional approximation. More precisely, let $\mathcal{F}_K(X)$ the space of finite linear combinations of the functions $\{K_x \,|\, x\in X\}$. The previous theorem leads to the following result:
\begin{proposition}[Interpolators in $\mathcal{H}_K$]
    Let $X$ be a finite subset of $\Omega$ and $f \in \mathcal{H}_K$. Then, the problem
    \begin{equation}
        \underset{g \in \mathcal{F}_K(X)}{\min}\,\|f-g\|_{\mathcal{H}_K}
    \end{equation}
    has a unique solution, given by $g = S_{K,X}(f)$, where $S_{K,X}$ is the orthogonal projection on $\mathcal{F}_K(X)$. Moreover, $S_{K,X}(f)$ is the interpolator of $f$ on $X$ with minimal $\mathcal{H}_K$ norm.
\end{proposition}
An direct consequence of this result is the important Representer Theorem, which applies to regularized regression:
\begin{proposition}[Representer Theorem]\label{prop:rep_th_pd_case}
    Let $X = (x_1, \dots, x_n) \in \Omega^n$ and $Y = (y_1, \dots, y_n) \in \R^n$. Then, the solution to the regularized regression problem
    \begin{equation}
        \underset{f \in \mathcal{H}_K}{\min}\,\sum_{i=1}^n (f(x_i) - y_i)^2 + \lambda \|f\|^2_{\mathcal{H}_K}
    \end{equation}
    lies in $\mathcal{F}_K(X)$. Moreover, it is given by
    \begin{equation}
        f = \sum_{j=1}^n \gamma_j K_{x_j}\,,
    \end{equation}
    where $\gamma = (\gamma_1, \dots, \gamma_n)$ is obtained as 
    \begin{equation}
        \gamma = (K_X + \lambda \Id_n)^{-1}Y \,.
    \end{equation}
\end{proposition}

\subsubsection{Case of a $\mathbb{P}$ conditionally positive definite kernel.}

Now, we extend the previous results to the case of conditionally positive definite kernels. This exposition is mainly based on \cite{auffray_pdcpker_2009,Wendland_2004}.
\begin{definition}[$\mathbb{P}$-CPD kernels, point-wise version]\label{defi:pcpd_ker}
    Let $K : \Omega  \times \Omega \rightarrow \R$ be a symmetric kernel and $\mathbb{P} \subset \R^\Omega$ be a finite-dimensional function space. The kernel $K$ is said to be $\mathbb{P}$-conditionally positive definite if, for any distinct points $x_1, \dots, x_N \in \Omega$, $c_1, \dots c_N \in \R$ satisfying
    \begin{equation}
        \forall p \in \mathbb{P},\, \sum_{i=1}^n c_i p(x_i) = 0 \,,
    \end{equation}
    the following inequality holds
    \begin{equation}
        \sum_{i,j = 1}^n c_i c_j K(x_i,x_j) \geq 0
    \end{equation}
    with an equality if and only if $c_1 = \dots = c_n = 0$.
\end{definition}
\begin{example}
    \begin{itemize}
        \item The TPS kernels \
    \end{itemize}
\end{example}
Let us define the measure space
\begin{equation}
    \mathcal{M}_{\mathbb{P}} \coloneqq \{\gamma \in \mathcal{M}(\Omega) \,|\, \forall p \in \mathbb{P},\gamma(p) = 0\}\,,
\end{equation}
as well as the function space, for a finite set $X$:
\begin{equation}
    \mathcal{F}_{K,\mathbb{P}}(X) \coloneqq \{K\star\gamma \,|\, \gamma \in \mathcal{M}_{\mathbb{P}}\cap \{\text{supp}(\gamma) = X\}\}\,,
\end{equation}
and 
\begin{equation}
    \mathcal{F}_{K,\mathbb{P}} \coloneqq \bigcap_{\underset{X \text{is finite}}{X \subset \Omega}} \mathcal{F}_{K,\mathbb{P}}(X) \,.
\end{equation}
By Definition \ref{defi:pcpd_ker}, the bilinear form 
\begin{equation}
    \langle \cdot, \cdot \rangle_{\mathcal{F}_{K,\mathbb{P}}} : \left(K \star \alpha = \sum_i \alpha_i K_{x_i},K \star \beta = \sum_j \beta_j K_{y_j}\right) \in \mathcal{F}_{K,\mathbb{P}}^2 \mapsto \sum_{i,j} \alpha_i \beta_j K(x_i,y_j) \in \R
\end{equation}
is positive definite on $\mathcal{F}_{K,\mathbb{P}}$.
To generalize Aronszajn's result in this setting \cite{auffray_pdcpker_2009}, we need the following definition:
\begin{definition}
    A vector space $\mathcal{N}$ equipped with a symmetric positive bilinear form $\langle \cdot, \cdot \rangle_\mathcal{N}$ is said to be a semi-Hilbert if $\mathcal{N}/\mathcal{K}$ is a Hilbert space, where $\mathcal{K}$ is the null-space of $\langle \cdot, \cdot \rangle_\mathcal{N}$.
\end{definition}
Now, the generalized Aronszajn's Theorem from \cite[Theorem 4.1]{auffray_pdcpker_2009} states:
\begin{theorem}
    Let $K$ be a $\mathbb{P}$-CPD kernel. Then, there exists a unique semi-Hilbert space $(\mathcal{H}_{K,\mathbb{P}},\langle \cdot, \cdot \rangle_{\mathcal{H}_{K,\mathbb{P}}})$ such that
    \begin{enumerate}
        \item The space $(\mathcal{F}_{K,\mathbb{P}},\langle \cdot, \cdot \rangle_{\mathcal{F}_{K,\mathbb{P}}} )$ is a pre-Hilbert subspace of $(\mathcal{H}_{K,\mathbb{P}},\langle \cdot, \cdot \rangle_{\mathcal{H}_{K,\mathbb{P}}})$.
        \item We have the direct sum decomposition $\mathcal{H}_{K,\mathbb{P}} = \mathbb{P} \oplus \mathcal{F}_{K,\mathbb{P}}$.
    \end{enumerate}
\end{theorem}
To study the interpolation problem through a $\mathbb{P}$-CPD kernel, we need to introduce the following definition:
\begin{definition}[$\mathbb{P}$-unisolvent sets]
    A finite set $E = \{\omega_1,\dots,\omega_l\} \subset \R^d$ of $l$ distinct points is said to be $\mathbb{P}$-unisolvent if the application
    \begin{equation}
        I_E: \begin{array}{ccl}
                \mathbb{P} &\to& \R^l \\
                p &\mapsto& (p(\omega_1), \dots, p(\omega_l))
            \end{array}
    \end{equation}
    is injective. If it is a bijection, $E$ is said to be a \textit{minimal} $\mathbb{P}$-unisolvent set. In that case remark that $\dim \mathbb{P} = l$.
\end{definition}
\begin{remark}
    If $X = \{ x_1, \dots, x_n \}$ contains a $\mathbb{P}$-unisolvent set and $(p_1,\dots,p_l)$ is a $\mathbb{P}$ basis, the application
    \begin{equation}
        P_X^T: \begin{array}{ccl}
                \R^n &\to& \R^l \\
                \gamma &\mapsto& \left( \sum_{k=1}^n \gamma_k p_i(x_k), i\in \llbracket 1,l\rrbracket \right)
            \end{array}
    \end{equation}
    is a surjection.
\end{remark}
\begin{example}
    Let us introduce $\mathcal{P}_m$ the space of polynomial functions with degree at most $m$.
    \begin{itemize}
        \item If $\mathbb{P} = \mathcal{P}_0$, then any non-empty set is $\mathcal{P}_0$-unisolvent, and the minimal ones are the singletons.
        \item If $\mathbb{P} = \mathcal{P}_1$, in $\R^d$, a set of points contains a $\mathcal{P}_1$-unisolvent set if and only if it is not contained in a hyperplane.
        \item If $\mathbb{P} = \mathcal{P}_m$ with $m \leq d-1$ , then a set $E$ fails to be $\mathcal{P}_m$-unisolvent if and only if there exists a nonzero polynomial $p \in \mathcal{P}_m$ such $p(\omega) = 0$, for all $\omega \in E$. The null set of $p$ is an algebraic hypersurface of Lebesgue measure zero, and the union of such sets as $p$ ranges over $\mathcal{P}_m$ still has zero Lebesgue measure.
    \end{itemize}
\end{example}
This allows to write the following result about interpolation in $\mathcal{H}_{K,\mathbb{P}}$:
\begin{proposition}\label{prop:exact_matching_cpd}
    Let $X$ be a finite subset of $\Omega$ containing a $\mathbb{P}$-unisolvent set. Then the problem
    \begin{equation}
        \underset{g \in \mathbb{P} \oplus \mathcal{F}_{K,\mathbb{P}}(X)}{\min}\,\|f-g\|_{\mathcal{H}_{K,\mathbb{P}}}
    \end{equation}
    has a unique solution that interpolates $f$ on $X$, denoted by $g = S_{K,\mathbb{P},X}(f)$. Moreover, $S_{K,\mathbb{P},X}(f)$ is the interpolator of $f$ on $X$ with minimal $\mathcal{H}_{K,\mathbb{P}}$ semi-norm.
\end{proposition}
A Representer Theorem, analogous to Proposition \ref{prop:rep_th_pd_case} can be proven:
\begin{proposition}\label{prop:reg_pcpd}
    Let $X = (x_1, \dots, x_n) \in \Omega^n$ be a collection of $n$ distinct points containing a $\mathbb{P}$-unisolvent set, and let $Y = (y_1, \dots, y_N) \in \R^n$. Consider the following regularized regression problem
    \begin{equation}
        \underset{f \in \mathcal{H}_{K,\mathbb{P}}}{\min}\sum_{i=1}^n (f(x_i) - y_i)^2 + \lambda \|f\|^2_{\mathcal{H}_{K,\mathbb{P}}}\,,
    \end{equation}
    where $\lambda >0$. Then, its solution lies in $\mathbb{P}\oplus\mathcal{F}_{K,\mathbb{P}}(X)$. 
    Let $(p_1, \dots, p_l)$ be a $\mathbb{P}$ basis, $K_X$ the $n \times n$ matrix with entries $K(x_i,x_j)$ and $P_X$ the $n \times l$ matrix with entries $p_j(x_i)$. The solution $f$ has the explicit form
    \begin{equation}
        f = \sum_{i=1}^n \gamma_i K_{x_i} + \sum_{j=1}^l \alpha_j p_j\,,
    \end{equation}
    where $\gamma = (\gamma_1, \dots, \gamma_n)$ and $\alpha = (\alpha_1, \dots, \alpha_l)$ are given by
    \begin{equation}
        \begin{cases}
            \gamma &=  (K_X + \lambda \Id_n)^{-1}(Y-P_X \alpha) \,,\\
            \alpha &= (P_X^T(K_X+\lambda \Id_n)^{-1}P_X)^{-1}P_X^T(K_X + \lambda \Id_n)^{-1}Y \,.
        \end{cases}
    \end{equation}
\end{proposition}

\subsection{Case of the Energy-Distance kernel, induced norm on $\mathcal{H}_{K,\mathbb{P}}$\,.}

The semi-norm defined on $\mathcal{H}_{K,\mathbb{P}}$ induces a semi-metric on $\R^n$ that depends on the interpolation points $X = (x_1, \cdots, x_n) \in \left(\R^d\right)^n$. More precisely, we can state the following proposition:
\begin{proposition}\label{prop:induced_norm}
    Let $K$ be the Energy-Distance kernel. Let $\mathbb{P}$ be a finite-dimensional space of polynomial functions from $\R^d$ to $\R$ that contains the set  $\mathcal{P}_0$ of constant polynomial functions. Let us denote $(p_1, \dots,p_l)$ a $\mathbb{P}$ basis. Let $X = (x_1, \dots, x_n) \in (\R^d)^n$ be a tuple of $n$ pairwise distinct points containing a $\mathbb{P}$-unisolvent set. Then the application
    \begin{equation}
        \begin{array}{c c}
            L_X: &\begin{array}{c l}
                \mathcal{H}_{K,\mathbb{P}}(X) \coloneqq \mathcal{F}_{K,\mathbb{P}}(X) \oplus \mathbb{P}  &\rightarrow \R^N \\
                v &\mapsto (v(x_1), \dots, v(x_n))
            \end{array}
        \end{array}
    \end{equation}
    is a linear isomorphism. Moreover, using the same matrices $K_X$ and $P_X$ as in \ref{prop:reg_pcpd} and given interpolation values 
    $Y = (y_1, \dots, y_n) \in \R^n$, there exists a unique couple $(\gamma_X(Y), \alpha_X(Y)) \in R^n \times \R^l$ such that 
    \begin{equation}\label{eq:param_RN_ed_ker}
        \begin{cases}
            Y = K_X \gamma_X(Y) + P_X \alpha_X(Y) \,,\\
            P_X^T \gamma_X(Y) = 0 \,.
        \end{cases}
    \end{equation}
    As a consequence, $L_X$ induces a semi-metric $g_{K,\mathbb{P},X}$ on $\R^n$ through
    \begin{equation}
        g_{K,\mathbb{P},X}(Y,Y) \coloneqq \gamma_X(Y)^T K_X \gamma_X(Y) \,.
    \end{equation}
    The null space of this semi-metric is exactly the set $\{P_X \alpha \,|\, \alpha \in \R^l\} \subset \R^n\,.$
\end{proposition}
\begin{proof}
    First, we prove that $\mathcal{H}_{K,\mathbb{P}}(X)$ is a $n$-dimensional space. More precisely, by defining the constraint set $\Omega_{X,\mathbb{P}} \coloneqq \{ \gamma \in \R^N \,|\, P_X^T \gamma = 0 \}$, we show that the application
    \begin{equation}
        \begin{array}{c c}
            I_X: &\begin{array}{c l}
                \Omega_{X,\mathbb{P}} \times \R^m   &\rightarrow \mathcal{H}_{K,\mathbb{P}}(X) \coloneqq \mathcal{F}_{K,\mathbb{P}}(X) \oplus \mathbb{P} \\
                (\gamma,\alpha) &\mapsto v(X,\gamma,\alpha) \coloneqq \sum_{i=1}^N \gamma_i K_{x_i} + \sum_{j=1}^m \alpha_j p_j 
            \end{array}
        \end{array}
    \end{equation}
    is a linear isomorphism. Remark that the application $I_X$ is surjective by construction. Now, let $(\gamma,\alpha) \in \Omega_{X,\mathbb{P}} \times \R^l $ such that $I_X(\gamma,\alpha) = 0$. Since the sum is direct, this implies $\sum_{i=1}^n \gamma_i K_{x_i} = \sum_{j=1}^l \alpha_j p_j = 0$. As $(p_1,\dots,p_l)$ is a $\mathbb{P}$ basis, we get $\alpha = 0$. Next, the equation $\sum_{i=1}^n \gamma_i K_{x_i} = 0$ implies:
    \begin{equation}
        \left\| \sum_{i=1}^n \gamma_i K_{x_i} \right\|_{\mathcal{H}_{K,\mathbb{P}}}^2 = \sum_{i,j = 1}^n \gamma_i \gamma_j K(x_i,x_j) = 0 \,.
    \end{equation}
    However, since $\mathcal{P}_0 \subset \mathbb{P}$, we have $\sum_i \gamma_i = 0$, so that the equality above holds if and only if $\gamma = 0$, proving that $I_X$ is a linear isomorphism. 
    Finally, since $P_X^T$ has rank $l$ as $X$ contains a unisolvent set, we obtain $\dim \Omega_{X,\mathbb{P}} = n - l$. Consequently, $\dim \mathcal{H}_{K,\mathbb{P}}(X) = n$. From \cite[Lemma 5.1,5.2]{auffray_pdcpker_2009}, this proves that $L_X$ is a linear isomorphism.
    As an additional corollary of the isomorphism $I_X$, there exists a unique couple $(\gamma_X(Y),\alpha_X(Y)) \in \Omega_{X,\mathbb{P}} \times \R^l$ such that Equation \eqref{eq:param_RN_ed_ker} is satisfied , given values $Y = (y_1, \dots, y_n) \in \R^N$. Consequently, the semi-metric $g_{K,\mathbb{P},X}$ is well-defined, and we get the equivalence:
    \begin{equation}
        g_{K,\mathbb{P},X}(Y,Y) = 0 \Leftrightarrow \gamma_X(Y) = 0 \Leftrightarrow \exists \alpha_X(Y) \in \R^l \text{ such that } Y = P_X \alpha_X(Y) \,.
    \end{equation}
\end{proof}
This provides a way to induce a non-degenerate metric on $\R^n$, although other choices are possible:
\begin{corollary}
    With the previous notations, if $\lambda >0$ the formula:
    \begin{equation}
        \Tilde{g}_{K,\mathbb{P},X,\lambda}(Y,Y) \coloneqq \gamma_X(Y)^T K_X \gamma_X(Y) + \lambda\alpha_X(Y)^T P_X^T P_X \alpha_X(Y)  \,
    \end{equation}
    defines a metric on $\R^n$.
\end{corollary}
When interpolating velocity fields, extending this choice by tensorization to a metric on $\R^{N \times d}$, using the same $\lambda$ for each dimension, ensures rotation invariance. Additionally, it naturally identifies the space of translations with Euclidean space, yielding an isotropic metric on $\mathbb{P}$ that better preserves geometric structure.

\subsection{Regularity and embeddings.}

In order to characterize the regularity of these spaces, let us introduce the Sobolev spaces, which are defined, for any $s \in \R$, as: 
\begin{equation}
    H^s \coloneqq \left\{ v \in \mathcal{S}'\, \bigg|\,\|v\|_{s} < \infty \right\}\,,
\end{equation}
where $\mathcal{S}'$ is the space of tempered distributions, $\|\cdot\|_{s}$ is the classical Sobolev norm, defined via
\begin{equation}
    \|v\|_{s,2}^2 \coloneqq \int_{\R^d} \left(1+|\xi|^2\right)^{s} |\hat{v}(\xi)|^2d\xi \,.
\end{equation}
Now, the classical Sobolev embedding theorem states \cite{adams_sob}:
\begin{theorem}
    Let $s > d/2$. Then 
    \begin{equation}
        H^s \hookrightarrow \mathcal{C}^0\,,
    \end{equation}
    where $\mathcal{C}^0$ is the space of continuous functions endowed with the uniform norm $\|\cdot\|_\infty$. Moreover, there exists $\lambda >0$ such that
    \begin{equation}
        H^s \hookrightarrow \mathcal{C}^{0,\lambda}\,.
    \end{equation}
    where $\mathcal{C}^{0,\lambda}$ is the set of Hölder continuous functions equipped with the norm $\|\cdot\|_{\lambda,\infty} = \|\cdot\|_\infty + |\cdot|_\lambda$, and $|\cdot|_\lambda$ is the Hölder semi-norm.
\end{theorem}
Now, following \cite{arcang_mult_min_spl_2004}, let us define the space
\begin{equation}
    \dot{H}^s \coloneqq \left\{ v \in \mathcal{S}'\, \bigg|\,|v|_{s} < \infty \right\}\,,
\end{equation}
where $\mathcal{S}'$ is the set of tempered distributions on $\R^d$, and $|\cdot|_{s}$ is the Sobolev semi-norm, defined as
\begin{equation}
    |v|_{s}^2 \coloneqq \int_{\R^d} |\xi|^{2s} |\hat{v}(\xi)|^2 d\xi \,.
\end{equation}
The $L^2$ Beppo-Levi spaces are defined as
\begin{equation}
    X^{m,s} \coloneqq \left\{ v \in \mathcal{D}' \, \bigg|\, \forall \alpha \in \mathbb{N}^d, |\alpha| = m, \partial^\alpha v \in \dot{H}^s\right\}\,,
\end{equation}
endowed with the semi-norm:
\begin{equation}
    |v|_{m,s} \coloneqq \left( \sum_{|\alpha| = m} |\partial^\alpha v|_s^2 \right)^{1/2}\,.
\end{equation}
Now, let us suppose the following condition on $m \in \mathbb{N}^*$ and $s \geq 0$:
\begin{equation}\label{eq:cond_s_m}
    -m + \frac{d}{2}<s<\frac{d}{2}\,.
\end{equation}
Then the following embedding theorem is valid \cite[Sec. I-2]{arcang_mult_min_spl_2004}:
\begin{theorem}\label{th:norm_pts_unisolvent}
Let us suppose that \eqref{eq:cond_s_m} holds.
    \begin{itemize}
        \item Let $E \subset \R^d$ be a finite subset, and define
        \begin{equation}
            \|v\|_{E,m,s} \coloneqq \left( \sum_{\omega\in E} |v(\omega)|^2 + |v|_{m,s}^2\right)^{1/2}\,.
        \end{equation}
        Then, if $E$ is a $\mathcal{P}_{m-1}$-unisolvent set, $\|\cdot\|_{E,m,s}$ defines a Hilbertian norm on $X^{m,s}$, whose topology is independent of $E$. Moreover, we have the continuous embedding:
        \begin{equation}
            X^{m,s} \hookrightarrow \mathcal{C}^0 \,.
        \end{equation}
        \item Let $\Omega$ be an open subset of $\R^d$ with Lipschitz boundary. Then the operator $R_{\Omega}$ of restriction to $\Omega$ is linear and continuous from $X^{m,s}$ to $H^{m+s}(\Omega)$.
    \end{itemize}
\end{theorem}
The following corollary allows for a different parametrization of a Hilbertian norm on $X^{m,s}$ when $m=1$:
\begin{corollary}[Parametrization through the mean on a design set]\label{coro:param_sum_translation}
    Suppose that \eqref{eq:cond_s_m} holds, with $m=1$. Let $X$ be a non-empty finite subset of $n$ points in $\R^d$, and define
    \begin{equation}
        \|v\|_{\Sigma X, m, s} \coloneqq \left(\left| \frac{1}{n} \sum_{x \in X}v(x) \ \right|^2 + |v|_{m,s}^2\right)^{1/2}\,.
    \end{equation}
    Then $\|\cdot\|_{\Sigma X, m, s}$ defines a Hilbertian norm on $X^{m,s}$, whose topology is independent of $X$. Moreover, it is equivalent to $\|\cdot\|_{\left\{\omega\right\}, m, s}$ from Theorem \ref{th:norm_pts_unisolvent}, for any $\omega \in \R^d$, so that the continuous embedding $X^{m,s} \hookrightarrow \mathcal{C}^0$ still holds with the norm $\|\cdot\|_{\Sigma X, m, s}$.
\end{corollary}
\begin{proof}
    Let $\omega \in \R^d$. Then the set $\{\omega\}$ is $\mathcal{P}_0$-unisolvent, so that we can apply Theorem \ref{th:norm_pts_unisolvent}. This means there exists $C_\omega > 0$ so that, for all $x \in X$:
    \begin{equation}
        |v(x)| \leq C_\omega \|v\|_{\left\{\omega\right\}, m, s}\,,
    \end{equation}
    so that 
    \begin{equation}
        \|\cdot\|_{\Sigma X, m, s} \leq (1+C_\omega^2)^{1/2} \|\cdot\|_{\left\{\omega\right\}, m, s}\,.
    \end{equation}
    Now let us prove that $(X^{m,s},\|\cdot\|_{\Sigma X, m, s})$ is a Banach space. Let $u_j$ be a Cauchy sequence in this space. Then:
    \begin{equation}
        \exists \Tilde{u} \in X^{m,s}, \underset{j}{\lim}\,|u_j - \Tilde{u}|_{m,s} = 0\,,
    \end{equation}
    \begin{equation}
        \exists \eta \in \R, \underset{j}{\lim}\,\frac{1}{n}  \sum_{x \in X}u_j(x) = \eta \,.
    \end{equation}
    Moreover, there exists $\lambda \in \R$ so that 
    \begin{equation}
        \frac{1}{n}  \sum_{x \in X}\Tilde{u}(x) + \lambda = \eta\,.
    \end{equation}
    This proves that $u_j$ converges to $\Tilde{u} + \lambda$ with respect to the $\|\cdot\|_{\Sigma X, m, s}$ norm, proving $(X^{m,s},\|\cdot\|_{\Sigma X, m, s})$ is a Banach space. We can apply the Banach Isomorphism Theorem \cite[II.6]{brezis} to obtain the equivalence of the two norms. The rest of the statement is a simple application of Theorem \ref{th:norm_pts_unisolvent}.
\end{proof}
Now, let us define:
\begin{definition}[$(m,s)$ splines kernels]\label{def:ms_splines}
    Let $m,s$ satisfying \eqref{eq:cond_s_m}, and denote $\nu \coloneqq 2m+2s-d$. The associated $(m,s)$ spline kernel is defined as
    \begin{equation}
        G_{m,s}(x) = \begin{cases}
            (-1)^{\lceil \nu \rceil} |x|^{2\nu}, &\nu \notin \mathbb{N}^*\,, \\
            (-1)^{\nu+1}|x|^{2\nu}\log|x|, &\nu \in \mathbb{N}^*\,.
        \end{cases}
    \end{equation}
\end{definition}
From \cite[Th 10.43]{Wendland_2004} and \cite[Sec. II]{arcang_mult_min_spl_2004}, we obtain the following characterization :
\begin{theorem}\label{th:carac_prkhs_bepo}
    Let $m,s$ satisfying \eqref{eq:cond_s_m}, and define $\mathcal{P}_{m-1}$ the space of $d$-variate polynomials of degree at most $m-1$. Then the kernel $G_{m,s}$ is $\mathcal{P}_{m-1}$-CPD. Moreover, the associated semi-Hilbert space is a Beppo-Levi space, given by:
    \begin{equation}
        \mathcal{H}_{G_{m,s},\mathcal{P}_{m-1}} = \mathcal{P}_{m-1} \oplus \mathcal{F}_{G_{m,s},\mathcal{P}_{m-1}} = X^{m,s}\,.
    \end{equation}
    Furthermore, $X^{m,s}$ being endowed with the semi-Hilbert norm $|\cdot|_{m,s}$, the semi-inner products of these spaces are the same, in the sense that, for all $f \in X^{m,s}$:
    \begin{equation}
        |v|_{m,s} = |v|_{\mathcal{H}_{G_{m,s},\mathcal{P}_{m-1}}}\,.
    \end{equation}
\end{theorem}
\subsection{A few possible choices for the Energy-Distance kernel.}\label{subsec:choices_of_P}
The Energy-Distance kernel $x \mapsto -|x|$ corresponds to an $(m,s)$ spline kernel, when $2m + 2s - d = 1$. This leads to several possible choices.
\subsubsection{The Energy-Distance as a pseudo-polynomial spline.}\label{subsubsec:ed_ker_m_is_1}
Choosing $m = 1$ and $s = \frac{d-1}{2}$, the kernel $G_{1,\frac{d-1}{2}}$ is referred to as a \textit{pseudo-polynomial} \cite[Sec II-3]{arcang_mult_min_spl_2004} or \textit{multiquadric} \cite{HARDY1990163} spline.
We can apply Proposition \ref{prop:induced_norm}, with $\mathbb{P} = \mathcal{P}_{0}$. The null space of the semi-metric $g_{K,\mathbb{P},X}$ is a one-dimensional set, equal to $\R \mathbbm{1}_N$, where $\mathbbm{1}_N = (1,\dots,1) \in \R^N$. In the context of velocity field interpolation, this corresponds to imposing no penalty on translations. Specifically, translations are the null-space of the tensorized semi-metric.

\subsubsection{Affine transformation invariance}
Choosing $m = 2$ and $s = \frac{d-3}{2}$ leads to $\mathbb{P} = \mathcal{P}_{1}$ and applying Proposition \ref{prop:induced_norm}, this time the null space of the associated metric $g_{K,\mathbb{P},X}$ is exactly the set of vectors $Y_0 = (f(x_i)+b)_{i\in[1,N]}$, where $f:\R^d\mapsto\R$ is a linear form and $b \in \R$. In the context of velocity field interpolation, this corresponds to imposing no penalty on affine deformations, which are the null-space of the tensorized semi-metric.
\subsubsection{The Energy-Distance as a Thin-Plate Spline kernel (TPS)}
The TPS kernels \cite{bookstein89tps,camionyounes2001} are defined as $(m,s)$ spline kernels for $s = 0$. Choosing $m = \frac{d+1}{2}$ if $d$ is odd, condition \eqref{eq:cond_s_m} is verified, and all the embedding theorems above are true.
Applying Proposition \ref{prop:induced_norm} with $\mathbb{P} = \mathcal{P}_{m-1}$, the set $\{P_X \alpha, \alpha \in \R^m\}$ is the null space of the induced semi-metric $g_{K,\mathbb{P},X}$. As in this paragraph $\mathbb{P} = \mathcal{P}_{\frac{d-1}{2}}$, for $d \geq 2$, this space contains, for example, all the vectors that can be written as $Y_0 = (f(x_i)+b)_{i\in[1,N]}$, where $f:\R^d\mapsto\R$ is a linear form and $b \in \R$. However, it contains many more if $d\geq 5$.

%% file: ed_ker_registration.tex
\section{Bi-Lipschitz homeomorphic matching using the Energy-Distance Kernel.}\label{SecTheNeedForRegularization}
We now show that regularization is necessary to obtain invertible and regular maps by proving that the so-called Landmark space endowed with the Energy distance kernel as the co-metric is not a complete Riemannian space. In other words, by solving the standard matching problem with the ED kernel, one can obtain non-invertible solutions of the variational problem.
\subsection{The variational problem on points}\label{sec:match_pb_without_tv}
In this section, we study the variational problem on the space of landmarks (i.e. $n$ distinct points), \textbf{without} any regularization on the vector field other than the norm induced by the kernel.
Specifically, we aim to match $n$ distinct points $(x_i)_{i = 1,\ldots,n} \in \mathbb R^d$ with $n$ distinct target points $(y_i)_{i = 1,\ldots,n} \in \mathbb R^d$. The displacement is modeled through velocity fields $v_t$, which are regularized through a norm $\|\cdot \|_{V_{m,s}}$. Here, $\|\cdot \|_{V_{m,s}}$ is any of the equivalent norms defined in Theorem \ref{th:norm_pts_unisolvent}, or in Corollary \ref{coro:param_sum_translation} if $m=1$. We assume that condition \eqref{eq:cond_s_m} is satisfied, and denote by $V_{m,s}$ the corresponding Beppo-Levi space endowed with the norm $\|\cdot \|_{V_{m,s}}$.
\begin{lemma}[Existence of trajectories]
    Let $v \in L^2([0,1],V_{m,s})$. Then, there exists a solution to the equation
    \begin{equation}\label{eq:lagr_traj_without_tv}
        \frac{d}{dt}x_i(t) = v_t (x_i(t))\,,
    \end{equation}
    with the initial condition $x_i(t) = x_i$.
\end{lemma}

\begin{proof}
    This is a simple application of the Caratheodory existence theorem. Indeed, the space $V_{m,s}$ is continuously embedded in $(\mathcal{C}^{0},\|\cdot\|_\infty)$ by Theorem \ref{th:norm_pts_unisolvent} or Corollary \ref{coro:param_sum_translation}. 
\end{proof}
\begin{remark}
    The trajectories are not necessarily unique for general $v_t$, as illustrated in \cite{ambrosio_crippa_2014}.
\end{remark}
However, they are uniformly continuous in time:
\begin{proposition}\label{prop:holder_cont_traj}
    Let $x_i$ be a solution to Equation \eqref{eq:lagr_traj_without_tv}. Then $t \in [0,1] \mapsto x_i(t)$ is $1/2$-Hölder continuous. 
\end{proposition}
\begin{proof}
    First, let us remark that any solution to Equation \eqref{eq:lagr_traj_without_tv} verifies, for all $t \geq 0$:
    \begin{equation}\label{eq:lag_traj_without_tv_integral}
        x_i(t) = x_i(0) + \int_0^t v_s(x_i(s))ds \,.
    \end{equation}
    This proves, using the Cauchy-Schwartz inequality and the continuous embedding $V_{m,s} \hookrightarrow \mathcal{C}^0$:
    \begin{equation}
        |x_i(t) - x_i(s)| \leq C_ {m,s}  \sqrt{|t-s|}\|v\|_{L_2([0,1],V_{m,s})}\,.
    \end{equation}
    for a constant $C_ {m,s} >0$.
\end{proof}
Now, we study the following problem:
\begin{equation}\label{eq:min_prob_without_tv}
    \min_{v \in L^2([0,1],V_{m,s})} \int_0^1 \| v_t\|^2_{V_{m,s}} dt \,,
\end{equation}
under the constraint that $\psi(t=1,x_i) = y_i$ for $(x_i)_{i = 1,\ldots,n} \in \mathbb R^d$, $n$ distinct points as well as $(y_i)_{i = 1,\ldots,n} \in \mathbb R^d$.
\begin{proposition}
    There exist solutions to the variational problem \eqref{eq:min_prob_without_tv}.
\end{proposition}
\begin{proof}
Let $v^k$ be a minimizing sequence for Problem \eqref{eq:min_prob_without_tv}. We can assume that $\|v^k\|_{L^2([0,1],V_{m,s})}$ is uniformly bounded in $k$, and by applying Proposition \eqref{prop:holder_cont_traj}, they are uniformly continuous. We can apply the Arzela-Ascoli Theorem, proving the existence of a limit velocity $v^{\infty}$ that is a solution to the matching problem \eqref{eq:min_prob_without_tv}. Now, passing to the limit in \eqref{eq:lag_traj_without_tv_integral}, the limiting curves $x_i^\infty$ are solutions to the ODE equation associated with $v^\infty$. Moreover, by uniform convergence, $x_i^{\infty} = y_i$. Finally, $v^\infty$ has minimal $\|\cdot\|_{L^2([0,1],V_{m,s})}$ norm by lower semi-continuity.
\end{proof}
\begin{remark}
Due to the non-uniqueness of solutions to the ODE, it may be possible to find non-diffeomorphic solutions. In other words, from $n$ distinct points, fewer points can be obtained during the evolution after a local collpase. Indeed, the non-uniqueness of solutions enables branching trajectories. 
\end{remark}
While the space $V_{m,s}$ is infinite dimensional, the solutions to Problem \eqref{eq:min_prob_without_tv} are sparse, in the following sense:
\begin{proposition}[Moment representation]
    Let $v \in L^2([0,1],V_{m,s})$ be a solution to Problem \eqref{eq:min_prob_without_tv}, and denote $X_t = (x_i(t), 1 \leq i \leq n)$. Let us denote $G_{m,s}$ as the corresponding $(m,s)$-spline kernel from Definition \ref{def:ms_splines}. Then, $v$ has the following form:
    \begin{equation}\label{eq:finite_spline_rep_without_tv}
        v_t(x) = \sum_{i=1}^n \gamma_i(t) G_{m,s}(x-x_i(t)) + p_t(x)\,,
    \end{equation}
    where $p_t \in \mathcal{P}_{m-1}$ at all times and $\gamma_i(t) \in \R^d$ for all $i$ verifies the constraint:
    \begin{equation}\label{eq:constraint_moments}
        \forall p \in \mathcal{P}_{m-1}, \sum_{i=1}^n \gamma_i(t) p(x_i(t)) = 0 \,.
    \end{equation}
\end{proposition}
\begin{proof}
    From Proposition \ref{prop:exact_matching_cpd}, for all times $t \geq 0$, the projection $S_{G_{m,s},\mathcal{P}_{m-1},X_t}(v_t)$ is written as in Equation \eqref{eq:finite_spline_rep_without_tv} with $\gamma_i(t)$ verifying the constraints \eqref{eq:constraint_moments}, interpolates $v_t$ on $X_t$, and verifies $$|S_{G_{m,s},\mathcal{P}_{m-1},X_t}(v_t)|_{m,s} \leq |v_t|_{m,s}\,,$$with equality if and only if $v_t = S_{G_{m,s},\mathcal{P}_{m-1},X_t}(v_t)$. This completes the proof. 
\end{proof}

\begin{example}[Two particles case]\label{ex:two_particules}
    Let us illustrate through a simple but instructive example the non-completeness of the landmark manifold endowed with this metric. In dimension $2$, we consider the Energy-Distance kernel as a $(1,\frac{1}{2})$-spline (see \ref{subsubsec:ed_ker_m_is_1}). We denote $V$ the associated Beppo-Levi space. Let $x_+ = (\frac{1}{2}r_0,0)$ and $x_- = (-\frac{1}{2}r_0,0)$ be two points in $\R^2$. We consider the target points $x_+^\varepsilon = (\frac{1}{2}\varepsilon,0)$ and $x_-^\varepsilon = (-\frac{1}{2}\varepsilon,0)$, where $r_0 > \varepsilon > 0$. Let us suppose that there exists $v \in L^2([0,1],V)$ solution to the corresponding minimal energy matching problem \eqref{eq:min_prob_without_tv}. Thanks to the symmetries of the problem, the constant translation part of the optimal velocity field is zero, and the optimal moments belong in $\R(1,0)$. This means that the flow only depends on a time-continuous moment $t \mapsto \gamma_t \in \R$, in the sense that the optimal velocity field $v_t$ has the following form on $\R^2$:
    \begin{equation}
        \forall t \in [0,1],x \in \R^2, \exists \gamma_t \in \R \, |\, v_t(x) =  \begin{pmatrix} \gamma_t \\ 0   \end{pmatrix}  \cdot \left( |x-x_+(t)| - |x-x_-(t)| \right)\,.
    \end{equation}
    Denoting $r_t \coloneqq |x_+(t)-x_-(t)|$, the energy at time $t$ is given by:
    \begin{equation}
        \|v_t\|_{V}^2 = \gamma_t^2 r_t\,.
    \end{equation}
    Through a simple reparametrization argument \cite[Lemme 6]{glaun2005}, the optimal velocity fields are such that $t \mapsto \|v_t\|^2_{V}$ is constant in time, equal to $\lambda \coloneqq\gamma_0^2 r_0$.
    Let us suppose for now that $t \mapsto \gamma_t$ and $t \mapsto r_t$ are smooth functions on $(0,1)$, strictly positive with non-zero derivatives. Then, as
    \begin{equation}
        \begin{cases}
            x_+(t) = x_+ - \int_0^t \gamma_s r_sds \,, \\
            x_-(t) = x_- + \int_0^t \gamma_s r_sds \,,
        \end{cases}
    \end{equation}
    we get:
    \begin{equation}
        r_t = r_0 - 2\int_0^t \gamma_s r_sds \,,
    \end{equation}
    so that:
    \begin{equation}
        \dot{r_t} = - 2\gamma_t r_t \,.
    \end{equation}
    Moreover, the following conservation equation is verified:
    \begin{equation}\label{eq:conservation_equation}
        \gamma_t^2 \dot{ r_t} + 2r_t\gamma_t \dot{ \gamma_t} = 0\,,
    \end{equation}
    so that the final differential equations read:
    \begin{equation}
        \begin{cases}
            \dot{\gamma_t} = \gamma_t^2 \,, \\
            \dot{r_t} = -2 \sqrt{\lambda}\sqrt{r_t} \,.
        \end{cases}
    \end{equation}
    The first equation is a well-known diverging equation, whose solution with initial conditions $\gamma_0$ exists on the open time interval $[0,\frac{1}{\gamma_0})$, and is written as:
    \begin{equation}
        \gamma_t = \frac{\gamma_0}{1-\gamma_0 t}\,.
    \end{equation}
    The corresponding $r_t$, thanks to the constant value of $\|v_t\|_V^2$, is given by:
    \begin{equation}
        r_t = r_0 (1-\gamma_0 t)^2 \,.
    \end{equation}
    Choosing $\gamma_0 = 1-\sqrt{\frac{\varepsilon}{r_0}}$ solves the matching problem, and justifies all the computations above. The final matching energy to match $(x_+,x_-)$ with $(x_+^\varepsilon,x_-^\varepsilon)$ is then equal to $\lambda_{r_0,\varepsilon} = \left( 1 - \sqrt{\frac{\varepsilon}{r_0}} \right)^2 r_0 $.
    Let us make a few remarks. First, $\lambda_{r_0,\varepsilon}$ has a limit as $\varepsilon \rightarrow 0+$, equal to $r_0$. This proves that the landmark manifold is not complete for this metric: there exists a sequence $v^n \in L^2([0,1],V)$ with uniformly bounded energy such that $\underset{n \rightarrow \infty}{\lim } \, x_+^{v_n}(1) = \underset{n \rightarrow \infty}{\lim } x_-^{v_n}(1) = 0$. 
   Due to the non-uniqueness of solutions to the flow, there exists a solution to the matching Problem \eqref{eq:min_prob_without_tv} under the constraints $x_+(1) = x_-$ and $x_-(1) = x_+$, even in dimension one. Indeed, the two points can collapse at $0$ in finite time. Then, by reversing the solution in time, one can exchange the position of the two particles.
    
    However, this situation is specific to this configuration, as a small perturbation in the positions induces a smooth path. 
    Below is an illustration of the deformation generated on $\R^2$, and of a matching between the perturbed landmarks $\left(x_+ + \left( \begin{smallmatrix} 0 \\ \varepsilon \end{smallmatrix} \right), x_- - \left( \begin{smallmatrix} 0 \\ \varepsilon \end{smallmatrix} \right)\right)$ and $\left(x_- + \left( \begin{smallmatrix} 0 \\ \varepsilon \end{smallmatrix} \right), x_+ - \left( \begin{smallmatrix} 0 \\ \varepsilon \end{smallmatrix} \right)\right)$, where $\varepsilon = 10^{-3}$.
    \begin{figure}[h!]
        \centering
        \includegraphics[width=0.5\linewidth]{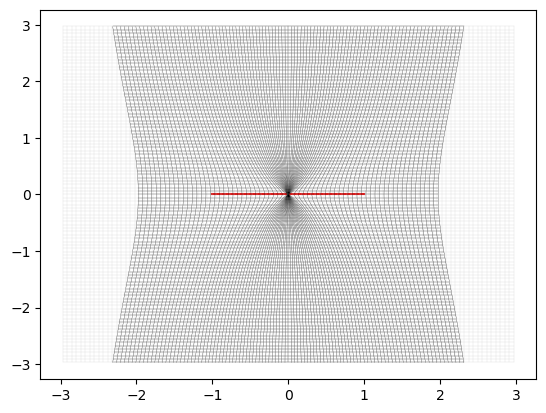}\includegraphics[width=0.45\linewidth]{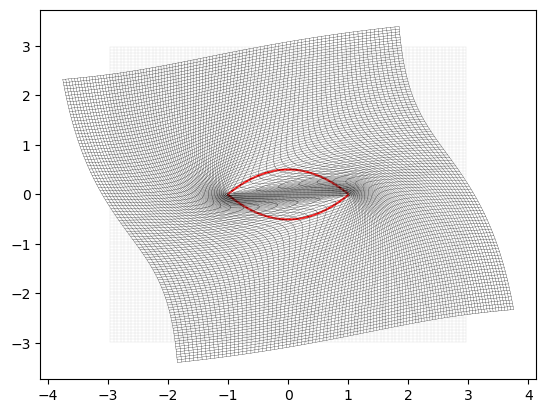}
        \caption{Obtained deformation map while matching two particles to $0$ (left), and between the landmarks $\left(x_+ + \left( \begin{smallmatrix} 0 \\ \varepsilon \end{smallmatrix} \right), x_- - \left( \begin{smallmatrix} 0 \\ \varepsilon \end{smallmatrix} \right)\right)$ and $\left(x_- + \left( \begin{smallmatrix} 0 \\ \varepsilon \end{smallmatrix} \right), x_+ - \left( \begin{smallmatrix} 0 \\ \varepsilon \end{smallmatrix} \right)\right)$ (right).}
        \label{fig:2_part_matching}
    \end{figure}
\end{example}

In summary, achieving homeomorphisms for the landmark problem requires regularization. We will examine a potential regularization approach in the following section. 

\subsection{Total variation and regularity}
In this section, we show that a bound on the total variation of the dual of the velocity field is sufficient to ensure Lipschitz regularity in the case of the ED kernel. 
Hereafter, $K$ denotes the Energy-Distance kernel, and $V_{m,s}$ is the corresponding Beppo-Levi space (see subsection \ref{subsec:choices_of_P}). In Example \ref{ex:two_particules}, $v_t$ is given by $v_t = K \star \Gamma_t$, where $\Gamma_t$ is the zero sum vector-valued measure defined by: 
\begin{equation}
    \Gamma_t \coloneqq \begin{pmatrix} \gamma_t \\ 0   \end{pmatrix} \left(\delta_{x_-(t)} - \delta_{x_+(t)}\right)\,.
\end{equation}
When matching $(x_+,x_-)$ to $(0_2,0_2)$, the total variation map $t \mapsto |\Gamma_t|_{TV}$ is not in $L^1([0,1])$. However, a regularity result can be obtained if we have such a bound on the total variation:
\begin{proposition}[Lipschitz regularity under $TV$ control]\label{prop:tv_control_existence_homeo_bilip}
    Let $t \mapsto v_t$ be a time-dependent velocity field written as:
    \begin{equation}
        v_t = K \star \Gamma_t\,,
    \end{equation}
    where $K$ is the Energy-Distance kernel and $\Gamma_t$ is a measure in $\mathcal{M}_{\mathcal{P}_{m-1}}$ such that:
    \begin{equation}
        \int_0^1 |\Gamma_t|_{TV} dt < +\infty \,.
    \end{equation}
    Then $v_t$ is Lipschitz for almost every $t \in [0,1]$ and there exists a bi-Lipschitz homeomorphism curve $t \mapsto \psi_t$ such that $\psi_t$ is the unique solution to the flow equation
    \begin{equation}
        \begin{cases}
            \partial_t \psi_t = v_t \circ \psi_t \, \\
            \psi_0 = \Id \,.
        \end{cases}
    \end{equation}
\end{proposition}
\begin{proof}
    Since $\Gamma_t \in L^1([0,1],|\cdot|_{TV})$, the quantity $|\Gamma_t|_{TV}$ is bounded for almost every $t \in [0,1]$. Thus, for any $x,y \in \R^d$ and almost every $t$, we have:
    \begin{align}
        | v_t(x) - v_t(y) | &= \left| \int (|x-z| - |y-z|) d\Gamma_t(z) \right|  \\
        &\leq \int |x-y| d|\Gamma_t|_{TV}(z) \\
        | v_t(x) - v_t(y) | & \leq |\Gamma_t|_{TV} |x-y| \,.
    \end{align}
    This shows that $v_t$ is Lipschitz for almost every $t \in [0,1]$, with a Lipschitz constant such that: 
    \begin{align}
        \int_0^1 \text{Lip}(v_t)dt < +\infty\,.
    \end{align}
    Under these conditions, the existence of the flow $\psi_t$ for all $t \in [0,1]$ is well known, applying the Cauchy-Lipschitz theorem \cite{ambrosio_crippa_2014}. We can prove that this map is also Lipschitz. Indeed, if $x,y \in \R^d$, then:
    \begin{align}
        |\psi_t(x) - \psi_t(y)| &= \left| x - y + \int_0^t v_s (\psi_s(x)) - v_s(\psi_s(y))ds \right| \\
        & \leq |x-y| + \int_0^t |v_s (\psi_s(x)) - v_s (\psi_s(y)) |ds  \\
        |\psi_t(x) - \psi_t(y)| &\leq |x-y| + \int_0^t |\Gamma_s|_{TV} |\psi_s(x) - \psi_s(y)| ds \,,
    \end{align}
    so that a simple application of the Gr\"onwall Lemma yields:
    \begin{equation}\label{eq:estim_homeo_lip_tv}
        |\psi_t(x) - \psi_t(y)| \leq |x-y| \exp\left( \int_0^t |\Gamma_s|_{TV} ds \right) \,.
    \end{equation}
    Finally, defining the inverse flow $\psi^{-1}$ by:
    \begin{equation}
        \partial_t (\psi_t^{-1}) = -v_{1-t} \circ \psi_t^{-1} \,, 
    \end{equation}
    we can apply the same analysis, proving that $\psi_t$ is a bi-Lipschitz homeomorphism.
\end{proof}
Example \ref{ex:two_particules} above shows that the MMD functional $\Gamma \in \mathcal{M}_{\mathcal{P}_0} \mapsto \langle \Gamma , K \star \Gamma \rangle \in \R$ is not coercive (in general) with respect to the total variation norm. However, in the case of a sum of distinct Dirac masses, a bound on the TV norm can be obtained that depends on the minimal distance between points.
\begin{proposition}\label{prop:tv_maj_if_no_collpase}
    Let $\sigma,r>0$, and let $B_{\sigma r}$ be the closed ball of radius $\sigma r$ centered at $0$ in $\R^d$. Consider the subspace of $\mathcal{M}_{\mathcal{P}_{m-1}}(B_{\sigma r})$ defined as:
    \begin{equation}
        \mathcal{M}_{\mathcal{P}_{m-1},n,r}(B_{\sigma r}) = \left\{ \Gamma = \sum_{i=1}^n \gamma_i \delta_{x_i} \in \mathcal{M}_{\mathcal{P}_{m-1}}(B_{\sigma r})\, \bigg|\, \underset{i\neq j}{\min }|x_i - x_j| \geq r \right\}\,.
    \end{equation}
    Then, there exists $\lambda_\sigma > 0$ such that, for all $\Gamma \in \mathcal{M}_{\mathcal{P}_{m-1},n,r}(B_{\sigma r})$ we have:
    \begin{equation}\label{eq:control_tv_rkhs_npart}
        |\Gamma|_{TV} \leq \lambda_\sigma \sqrt{\frac{n}{r}} \|\Gamma\|_K \,,
    \end{equation}
    where $\|\Gamma\|_K = \langle \Gamma, K \star \Gamma \rangle$.
\end{proposition}

\begin{proof}
    Let us denote $L_{n,r} = \left\{ X = (x_1, \dots, x_n) \in \left(\R^d\right)^n \, \bigg|\, \underset{i\neq j}{\min }\,|x_i - x_j| \geq r\right\}$.
    First, remark that if $\Gamma \in \mathcal{M}_{\mathcal{P}_{m-1},n,r}(B_{\sigma r})$, then the scaled measure $d_{\lambda \#} \Gamma$, where $d_\lambda : x \mapsto \lambda x$, belongs to $\mathcal{M}_{\mathcal{P}_{m-1},n,\lambda r}(B_{\lambda \sigma r})$.
    
    Next, define $K_X$ the matrix with entries $K(x_i-x_j)$ and consider the linear space:
    \begin{equation}
        C_X = \left\{ \gamma = (\gamma_1, \dots, \gamma_N) \in \R^n \, \bigg|\, \forall p \in \mathcal{P}_{m-1}, \sum_{i=1}^n \gamma_i p(x_i) = 0 \right\}\,.
    \end{equation}
    Since the application $ \gamma \in C_X \cap \mathcal{S}_{n-1} \mapsto  \langle \gamma,K_X \gamma \rangle \in \R$ is continuous and takes strictly positive values, since $C_X \cap \mathcal{S}_{n-1}$ is compact, there exists $\rho_X > 0$ such that:
    \begin{equation}\label{eq:vp_maj_X}
        \forall \gamma \in C_X, \rho_X |\gamma|_2^2 \leq \langle \gamma,K_X \gamma \rangle \,,
    \end{equation}
    where $|\cdot|_2$ denotes the Euclidean distance in $\R^n$. This proves:
    \begin{equation}
        \forall \gamma \in C_X,  |\gamma|_2 \leq \sqrt{\frac{1}{\rho_X}} \sqrt{\langle \gamma,K_X u \rangle} \,.
    \end{equation}
    Now, for $\sigma,r>0$, define the set:
    \begin{equation}
        \mathcal{S}(r,\sigma) \coloneqq L_{n,r} \cap (B_{\sigma r })^n \,.
    \end{equation} 
    The mapping $\rho : X \mapsto \rho_X$ is continuous on $\mathcal{S}(r,\sigma)$, which is a compact set. Therefore, there exists $\rho_{r}(\lambda) > 0$ such that for all $X \in \mathcal{S}(r,\sigma)$, $\rho_{r}(\sigma) \leq \rho_X$. Now, let us prove, given $\sigma>0$, that $r \mapsto \rho_{r}(\sigma)$ is 1-homogeneous. From the equalities, for $r >0$: $\mathcal{S}(r,\sigma) = r\mathcal{S}(1,\sigma)$ and $K_{r X} = r K_{X}$ for any collection of distinct points $X$, we get the equivalences:
    \begin{align}
        &\forall X \in \mathcal{S}(1,\sigma),\forall \gamma \in C_X, \rho_{1}(\sigma) |\gamma|_2^2 \leq \langle \gamma, K_X \gamma \rangle \\
        &\Leftrightarrow \forall X \in \mathcal{S}(1,\sigma),\forall \gamma \in C_{r X}, \rho_{1}(\sigma) |\gamma|_2^2 \leq \frac{1}{r} \langle \gamma, K_{r X} \gamma \rangle \\
        &\Leftrightarrow \forall Y \in \mathcal{S}( r,\sigma),\forall \gamma \in C_Y, r \rho_{1}(R) |\gamma|_2^2 \leq \langle \gamma, K_Y \gamma \rangle\,.
    \end{align}
    Thus, for any $X \in L_{n,r} \cap (B_{\sigma r})^n =\mathcal{S}(r,\sigma)$ and $\gamma \in C_X$, we have:
    \begin{equation}
        r\rho_1(\sigma)|\gamma|_2^2 \leq \langle \gamma, K_X \gamma \rangle \,.
    \end{equation}
    Now, denoting $\lambda_\sigma \coloneqq \rho_1(\sigma)^{-1/2}$, we get:
    \begin{equation}\label{eq:control_L2_norm_rkhs}
        \forall X \in  L_{n,r} \cap (B_{\sigma r})^n, \forall \gamma \in C_X, |\gamma|_2 \leq \lambda_\sigma \sqrt{\frac{1}{r}}\sqrt{\langle \gamma,K_X \gamma \rangle} \,.
    \end{equation}
    The result follows from the fact that for any $\Gamma \in \mathcal{M}_{\mathcal{P}_{m-1},n,r}(B_{\sigma r})$, there exists $X = (x_1, \dots,x_n) \in L_{n,r} \cap (B_{\sigma r})^n$ and $\gamma \in C_X$ such that:
    $
        \Gamma = \sum_{i=1}^n \gamma_i \delta_{x_i}\,,
    $
    and by remarking that $|\Gamma|_{TV} = \sum_{i=1}^n |\gamma_i| \leq \sqrt{n}|\gamma|_2$ and $\|\Gamma\|_K = \sqrt{\langle \gamma , K_X \gamma \rangle}$.
\end{proof}
\begin{remark}
    Note that this result immediately generalizes to vector-valued measures. If $\Gamma$ is written as:
    $
        \Gamma = \sum_{i=1}^n \gamma_i \delta_{x_i} \,,$
    where $\gamma_i = (\gamma_{i,1},\dots,\gamma_{i,d}) \in \R^d$ for all $1 \leq i \leq n$, then:
    \begin{equation}
        |\Gamma|_{TV} = \sum_{i=1}^n \sqrt{\sum_{k=1}^d |\gamma_{i,k}|^2} \leq \lambda_{\sigma}\sqrt{\frac{n}{r}} \sqrt{ \sum_{k=1}^d \langle \gamma_{\cdot,k} , K_X \gamma_{\cdot,k} \rangle}\,.
    \end{equation} 
\end{remark}

\begin{example}
    In the limit $\varepsilon \rightarrow 0$ in Example \ref{ex:two_particules},  the solution moment measure $\Gamma_t$ satisfies $|\Gamma_t|_{TV} = \frac{1}{1-t}$ and $\|\Gamma_t\|_K = r_0$ for all times $t \in [0,1)$. Additionally, the distance between the two particles at time $0\leq t < 1$ is given by $r_t = r_0(1-t)^2$, so that $\Gamma_t$ is in the measure set $\mathcal{M}_{\mathcal{P}_{0},2,r_t}(B_{r_t}) = \mathcal{M}_{\mathcal{P}_{0},2, r_0(1-t)^2}(B_{r_0(1-t)^2})$. Applying Proposition \ref{prop:tv_maj_if_no_collpase} to $\Gamma_t$ provides 
    \begin{equation}
        |\Gamma_t|_{TV} \leq  \lambda_{1}\sqrt{\frac{2}{r_0 (1-t)^2}r_0} = \frac{\sqrt{2}\lambda_{1} }{1-t} \,,
    \end{equation}
    corresponding with the computations, so that the estimate \eqref{eq:control_tv_rkhs_npart} seems sharp.
\end{example}

\subsection{Two different regularizations ensuring bi-Lipschitz homeomorphisms}

To define models that guarantee bi-Lipschitz homeomorphisms, we propose two different approaches, which both lead to a control on the TV norm of $\Gamma$. The first model consists in directly adding the total variation in the formulation. This model can be applied to any type of data, such as clouds of points, measures, etc... The second model is tailored to probability measures and is easier to solve numerically.

\subsubsection{TV regularization on the dual space}
Let us define the measure space:
\begin{equation}\label{EqNormsTV_ED}
    \mathcal{B}_{0,TV}\coloneqq \left\{ \Gamma \in \mathcal{M}_0^d(\Omega) \,|\, |\Gamma|_{TV} < +\infty \right\} \,.
\end{equation}
Similarly, we define $\mathcal{B}_{TV}$ without the total mass constraint.
 Here, the total variation norm $|\cdot|_{TV}$ is defined through the duality with $(\mathcal{C}^0(\Omega),\|\cdot\|_{\infty})$ the space of bounded continuous function on $\Omega \subset \R^d$, a compact bounded domain.
In this section, we consider the energy: 
\begin{equation}\label{eq:mixed_model}
    E^{\beta} : \Gamma \in L^2([0,1],\mathcal{B}_{TV}) \mapsto  \mathcal{L}(\psi_\Gamma(1) _\# \mu_0 , \nu) +  \beta\int_0^1 |\Gamma_t|_{TV}^2dt \,,
\end{equation}
where $\psi_\Gamma(1)$ is the flow at time $1$ defined by $\psi_\Gamma(0) = \Id$ and $\partial_t \psi_{\gamma_t} = v_t \circ \psi_{\gamma_t}$ with $v_t = K \star \gamma_t$.
The goal of this subsection is to prove, provided that the loss $\mathcal{L}$ is lower semi-continuous for the weak-* topology, the following existence result:
\begin{theorem}\label{th:matching_tv_regu}
    The matching problem
    \begin{equation}
        \underset{\Gamma \in L^2([0,1],\mathcal{B}_{TV})}{\inf}  E^{\beta}(v) = \mathcal{L}(\psi_\Gamma(1) _\# \mu_0 , \nu) +  \beta \int_0^1 |\Gamma_t|_{TV}^2dt
    \end{equation}
    admits a solution in $L^2([0,1],\mathcal{B}_{TV})$. Moreover, the matching solution $\psi_{\Gamma}(1)$ is bi-Lipschitz homeomorphism.
\end{theorem}
First, let us state a few facts:
\begin{proposition}
    Bounded balls in $\mathcal{B}_{TV}$ are compact for the weak-* topology.
\end{proposition}
\begin{proof}
    Since $|\cdot|_{TV} \leq \frac{1}{\beta}\|\cdot\|_{\mathcal{B}_{TV}} $, bounded balls in $\mathcal{B}_{TV}$ are bounded for the total variation norm, so that they are compact for the weak-* topology by the Banach-Alaoglu theorem. 
\end{proof}
Next, the space $L^2([0,1],(\mathcal{M}_0^d(\Omega),|\cdot|_{TV}))$ is defined through the duality with $L^2([0,1],(\mathcal{C}^0(\Omega),\|\cdot\|_{\infty}))$. We still refer to the induced topology as the weak-* topology, and we get the result:
\begin{proposition}\label{prop:L_2_comp_weak_topo}
    Bounded balls in $L^2([0,1],(\mathcal{M}^d(\Omega),\|\cdot\|_{\mathcal{B}_{TV}}))$ are compact for the weak-* topology.
\end{proposition}
\begin{proof}
    Same proof as above.
\end{proof}
\begin{proposition}
    The map $\mu \in \mathcal{M}(\Omega) \mapsto \langle \mu , K \star \mu \rangle$ is continuous for the weak-* topology.
\end{proposition}
\begin{proof}
    If $\Omega$ is compact, then $K \star \mu$ is bounded by $|\mu|_{TV}\operatorname{diam}(\Omega)$ for any $\mu$. The identity $\langle \mu, K \star \mu \rangle - \langle \nu, K \star \nu \rangle = \langle \mu - \nu, K \star (\mu - \nu) \rangle + 2\langle \nu, K \star (\mu - \nu)\rangle$, combined with the Cauchy-Schwarz inequality,  proves the result. 
\end{proof}
Let us consider the flow application
\begin{equation}
        \psi(1): \begin{array}{ccl}
                L^2([0,1],\mathcal{B}_{TV}) &\to& \mathcal{C}^0(\Omega,\R^d) \\
                \Gamma &\mapsto& \psi_\Gamma(1)\,.
            \end{array}
    \end{equation}
\begin{proposition}\label{ThContinuityOfTheFlow}
    The application $\psi(1)$ is continuous from $L^2([0,1],\mathcal{B}_{TV})$ to $(\mathcal{C}^0(\Omega,\R^d),\|\cdot\|_\infty)$.
\end{proposition}
\begin{proof}
    Let $(\Gamma_k) \in L^2([0,1],\mathcal{B}_{TV})^\mathbb{N}$ such that $\Gamma_k \underset{k \rightarrow \infty}{\longrightarrow} \Gamma$ for the $\mathcal{B}_{TV}$ topology.
    Denote $\psi^k_s$ and $\psi_s$ the associated flows, corresponding to velocity fields $v^k_s$ and $v_s$. These flows exist by Proposition \ref{prop:tv_control_existence_homeo_bilip}, as $|\cdot|_{TV} \leq \frac{1}{\beta}\|\cdot\|_{\mathcal{B}_{TV}} $. Let $x \in \Omega$. Then:
    \begin{align}
        |\psi^k_t(x) - \psi_t(x)| &= \left| \int_0^t v^k_s(\psi^k_s(x)) -  v_s(\psi_s(x)) ds  \right| \\
        & \leq \int_0^t |v^k_s(\psi^k_s(x)) -  v^k_s(\psi_s(x))| ds + \left| \int_0^t v^k_s(\psi_s(x)) -  v_s(\psi_s(x)) ds  \right| \,.
    \end{align}
For the first term:
\begin{align}
    \int_0^t |v^k_s(\psi^k_s(x)) -  v^k_s(\psi_s(x))| ds & \leq \int_0^t \text{Lip}(v^k_s) |\psi^k_s(x) - \psi_s(x)|ds \\
    & \leq \frac{1}{\beta}\|\psi^k_s - \psi_s\|_{\infty} \int_0^t \|\Gamma^k_s\|_{\mathcal{B}_{TV}} ds \,.
\end{align}
For the second term, we define
\begin{equation}
    m_x(\Gamma) \coloneqq \int_0^t K\star \gamma_s (\psi_s(x))ds \,.
\end{equation}
which is a continuous linear application from $\mathcal{B}_{TV}$ to $\mathbb{R}^d$. From estimates \eqref{eq:estim_homeo_lip_tv}, the family $\{m_x ; x \in \Omega\}$ is equicontinuous. Using the Arzelà-Ascoli theorem, this ensures uniform convergence, in the sense
\begin{equation}
    \underset{k \rightarrow \infty}{\lim} \lambda_k \coloneq \underset{x \in \Omega}{\sup} \left| m_x(v) - m_x(v^k) \right| = 0 \,.
\end{equation}
We get
\begin{equation}
    \|\psi^k_s - \psi_s\|_{\infty} \leq \lambda_k + \frac{1}{\beta} \|\psi^k_s - \psi_s\|_{\infty} \int_0^t \|\Gamma^k_s\|_{\mathcal{B}_K^\lambda} ds\,.
\end{equation}
Thus, an application of Gr\"onwall's Lemma ends the proof.
\end{proof}
Then, Theorem \ref{th:matching_tv_regu} is a direct corollary of this proposition and Proposition \ref{prop:L_2_comp_weak_topo}. 

\vspace{0.2cm}

\noindent\textbf{Representation theorem. }
The regularization with the TV norm promotes sparse solutions to the variational problem. In this case, a Representer Theorem similar to \ref{prop:reg_pcpd} can be established: certain solutions to the corresponding minimization problem can be expressed as linear combinations of extremal points of the unit ball (see \cite{duvalboyer2018representertheoremsconvexregularization,Bredies_2019}).
First, let us consider a soft interpolation problem:
\begin{proposition}\label{prop:repr_total_var}
    Let $\Omega$ be a bounded subset of $\R^d$ and $X = (x_1, \dots, x_n) \in \left( \Omega \right)^N$ be a tuple of $n$ distinct points in $\R^d$. Given values $Y = (y_1,\dots, y_n)$, consider the problem
    \begin{equation}
        \underset{(\Gamma, p) \in \mathcal{M}_{0} \times \mathcal{P}_{0} }{\inf} \sum_{i=1}^n |K\star \Gamma(x_i) + p(x_i) - y_i|^2 + |\gamma|_{TV}\,
    \end{equation}
    admits a sparse solution, i.e. $\Gamma$ can be written as
    \begin{equation}
        \Gamma = \sum_{j=1}^m \alpha_j \delta_{z_j}\,,
    \end{equation}
    where $m \leq Nd + d$, $z_j \in \Omega$ for all $j$ and $\sum_{j=1}^m \alpha_j = 0$.
\end{proposition}
\begin{proof}[Proof of Proposition \ref{prop:repr_total_var}]
Consider an extension, denoted $\tilde K$ of the map $K \mathcal{M}_{0,TV} \to \mathcal{C}^0(X)$ defined by $\mu \mapsto K\star \mu$ to the space of measures $\mathcal{P}_{TV}$. Then, we apply \cite[Theorem 3.3]{Bredies_2019}. The function $z,p \in \Omega^n \times \mathcal{P}_0 \mapsto  \sum_i |z_i + p(x_i) - y_i|^2$ is coercive, convex, and l.s.c for the usual topology in the Euclidean space. Moreover, the operator $\mathcal{A} : \gamma,p \in \mathcal{M} \times \left(\R^d\right)^n \times \mathcal{P}_0 \mapsto (\tilde K\star\Gamma(x_i) + p(x_i),1\leq i \leq n) \in \left(\R^d\right)^n $ is a linear, continuous and surjective. In addition, we add the constraint $\mu(\Omega) = 0 \in \R^d$. By \cite[4.1,3]{Bredies_2019}, there exists a solution which is a linear combination of $nd + d$ Dirac masses (the extremal points of the TV ball of $\mathcal{M}(\Omega,\R^d)$).
\end{proof}
For more details on sparse representation theorems, see \cite[4.1,3]{Bredies_2019} and \cite{duvalboyer2018representertheoremsconvexregularization}.
\begin{remark}[About the proof]
    To directly apply the result from \cite{Bredies_2019}, the extremal points of the ball of $\mathcal{M}_0(\Omega,\R^d)$ need to be characterized. These extremal points are sums of at most $d+1$ Dirac masses, as shown in the Appendix. However, this approach would give an $Nd(d+1)$ upper bound on the number of Dirac masses. Thus, a more effective approach is to incorporate the total volume constraint directly into the functional.
\end{remark}
\begin{remark}
    The above result does not provide information about the localization of the Dirac masses $\delta_{z_i}$. Recall that the RKHS representation result guarantees $N$ Dirac masses for an optimal solution, which stands in sharp contrast with the result we have $(N+1)d$. However, this bound is only an upper estimate, and it remains unclear whether it is tight.
\end{remark}

\subsubsection{$L^2$ regularization on the momentum.}

In this section, we combine the regularity properties of the previous model with the simpler representation of the model from Section \ref{sec:match_pb_without_tv}. Specifically, let $\mu \in \mathcal{P}(\Omega)$, where $\Omega$ is a bounded domain. If $\gamma \in L^2(\mu,\R^d)$, then the measure $\Gamma \coloneqq \gamma \mu$ is finite.
 \begin{lemma}
    Let $\mu \in \mathcal{P}(\Omega)$, where $\Omega$ is a bounded domain, and let $\gamma \in L^2(\mu,\R^d)$. Define $\Gamma \coloneqq \gamma \mu$. Then, the quantity $\langle \Gamma, K\star\Gamma\rangle$ is well-defined. Furthermore, the function $K \star \Gamma$ is also in $L^2(\mu,\R^d)$, with the inequalities:
    \begin{equation}
     \|K \star \Gamma\|_{L^2(\mu)} \leq \operatorname{diam}(\Omega)\|\gamma\|_{L^2(\mu)}\,.
    \end{equation}
    \begin{equation}
     |\Gamma|_{TV} \leq \|\gamma\|_{L^2(\mu)} \,.
    \end{equation}
\end{lemma}

 \begin{proof}
     Since $\Omega$ is a bounded domain, the Energy-Distance kernel is bounded by diam$(\Omega)$ on the support of $\mu$. This proves that $\|K \star \Gamma\|_{L^2(\mu)} \leq \operatorname{diam}(\Omega)\|\gamma\|_{L^2(\mu)}$ and $\langle \Gamma, K\star\Gamma\rangle \leq \operatorname{diam}(\Omega) \|\gamma\|_{L^2(\mu)}^2$.
     For the second part, we apply the Cauchy-Schwarz inequality:
     \begin{equation}
         |\Gamma|_{TV} = \int |\gamma|d\mu \leq \left( \int |\gamma|^2 d\mu\right)^{1/2} = \|\gamma\|_{L^2(\mu)} \,.
     \end{equation}
 \end{proof} 
Now, if $\mu_\cdot$ is an absolutely continuous measure curve \cite[Sec I.1]{ambrosio2005gradient} in $AC^2([0,1],\mathcal{P}(\Omega))$, we define the $\mu_{\cdot}$ dependant function space:
\begin{equation}
    L^2([0,1],L^2(\mu_t)) = \{ \gamma : t \mapsto \gamma_t \in L^2(\mu_t) \, |\, \int_0^1 \|\gamma_t\|_{L^2(\mu_t)}^2dt < + \infty \}\,.
\end{equation}
Define the vector-valued measure curve $\Gamma : t \mapsto \Gamma_t \coloneqq\gamma_t \mu_t $, and the associated velocity $v_t \coloneqq K \star \Gamma_t  \in L^2(\mu_t)$. Then, the system
\begin{equation}\label{eq:sys_transp_momentum}
    \begin{cases}
        \partial_t\mu_t + \nabla \cdot (v_t \mu_t) = 0 \,, \\
        v_t = K \star\Gamma_t \,,\\
        \Gamma_t = \gamma_t \mu_t \,, \\
        \gamma_t \in L^2(\mu_t) \,,
    \end{cases}
\end{equation}
where the first equation holds in the sense of distributions \cite[Chap. 8]{ambrosio2005gradient}, is well defined, and admits solutions $(\mu_t,\gamma_t) \in AC^2([0,1],\mathcal{P}(\Omega)) \times L^2([0,1],L^2(\mu_t))$. 
Now, we define the matching problem:
\begin{equation}\label{eq:prob_reg_L2_mom}
\underset{(\mu,\gamma) \text{ solves } \eqref{eq:sys_transp_momentum} }{\inf}
    E^\lambda(\mu,\gamma) = \mathcal{L}(\mu_1,\nu) + \int_0^1 \langle \Gamma_t, K \star \Gamma_t \rangle dt + \int_0^1 \|\gamma_t\|_{L^2(\mu_t)}^2dt \,,
\end{equation}
where $(\mu,\gamma) \in AC^2([0,1],\mathcal{P}(\Omega)) \times L^2([0,1],L^2(\mu_t))$ is a solution to the system \ref{eq:sys_transp_momentum}, $\mu_1$ is the measure $\mu_t$ at time $1$, and $K$ is the Energy-Distance kernel. Here, $\nu$ is a reference target measure, and $\mathcal{L}$ is any weakly-continuous loss between $\mu_1$ and $\nu$ (see Example \ref{ex:loss_examples}).
Let us start with the following lemma, which allows to control the growth of the support of $\mu_t$ over time:
\begin{lemma}[Support growth bound]\label{lemma:time_dep_bound}
    Let $(\mu,\gamma)$ be a solution to the system \eqref{eq:sys_transp_momentum}. Let us suppose that $\operatorname{supp}(\mu_0)$ is bounded, and that:
    \begin{equation}
        \int_0^1 \|\gamma_t\|^2_{L^2(\mu_t)}dt  \leq C\,,
    \end{equation}
    for some positive constant $C$. Then, for all $t \in [0,1]$:
    \begin{equation}
        \operatorname{diam}(\operatorname{supp } \mu_t) \leq \operatorname{diam}(\operatorname{supp } \mu_0) \exp{\left(\left(\int_0^t \|p\|^2_{L^2(\mu_t)}dt\right)^{1/2}\right)}\,.
    \end{equation}
\end{lemma}
\begin{proof}
    For almost every $t \in (0,1)$, for $\mu_t$-almost any $x$, we have:
    \begin{equation}
        |v_t (x)| \leq \operatorname{diam}(\operatorname{supp } \mu_t) \|p\|_{L^2(\mu_t)}\,,
    \end{equation}
    so that
    \begin{equation}
        \frac{d}{dt} \operatorname{diam}(\operatorname{supp } \mu_t) \leq \operatorname{diam}(\operatorname{supp }\mu_t)\|p\|_{L^2(\mu_t)}\,. 
    \end{equation}
    Using Grönwall's Lemma and the Cauchy-Schwarz inequality, we obtain 
    \begin{align}
        \operatorname{diam}(\operatorname{supp } \mu_t) &\leq \operatorname{diam}(\operatorname{supp } \mu_0) \exp{\left(\int_0^t \|p\|_{L^2(\mu_t)}dt\right)} \\
        &\leq \operatorname{diam}(\operatorname{supp } \mu_0) \exp{\left(\left(\int_0^t \|p\|^2_{L^2(\mu_t)}dt\right)^{1/2}\right)}\,.
    \end{align}
\end{proof}
The following result completely characterizes solutions to the system \eqref{eq:sys_transp_momentum} in the case of atomic initial measures with finite $E^\lambda$ energy, and provides a straightforward parametrization of the flow.
\begin{proposition}[$L^2$ parametrization]\label{prop:L2_mom_reparam}
    Let $X = (x_1, \dots, x_N) \in \left(\R^d\right)^n$, $p \in L^2\left([0,1],\left( \R^d \right)^n\right) $. Then, there exists an absolutely continuous measure curve $\mu$ written as 
    \begin{equation}\label{eq:def_mu_t}
            \mu_t = \sum_{i=1}^n \frac{1}{n}\delta_{x_i(t)}\,,
    \end{equation}
    such that, defining $\gamma_t \in L^2(\mu_t) : x_i(t) \mapsto p_t^i$, $(\mu,\gamma)$ is a solution to the system \eqref{eq:sys_transp_momentum} and such that the $x_i(t)$ are distinct at all times.
    Furthermore, provided initial distinct points $X = (x_1, \dots, x_n) \in \left(\R^d\right)^n$, the application
    \begin{equation}
        p \in L^2\left([0,1],\left( \R^d \right)^n\right) \mapsto \mu_t \in \mathcal{P}(\R^d)
    \end{equation} 
    is weakly continuous for all $t \in [0,1]$, and the support of $\mu_t$ is uniformly bounded in time.
    Conversely, if $\mu_t$ is an absolutely continuous curve written as in \eqref{eq:def_mu_t}, where the points are distinct at all times, with support contained in a uniform compact set $\Omega$, then there exists $p \in L^2\left([0,1],\left( \R^d \right)^n\right)$ such that $\mu_t = \mu_t(p)$ as constructed in the previous statement.
\end{proposition}
\begin{proof}
    Let $\mu$ be a probability measure with finite support $(x_1, \dots, x_n)$ and $p : x_i \mapsto p_i \in L^2(\mu)$. Define the velocity fields $v = K \star (p\mu) $. Then, $v$ is Lipschitz, with a Lipschitz constant bounded by $\|p \|_{L^2(\mu)}$, which is independent from the support of $\mu$. This ensures, by the Cauchy-Lipschitz Theorem, that the Lagrangian equations
    \begin{equation}\label{eq:lag_eq_particl}
        \begin{cases}
            \dot{x_i}(t) = v_t(x_i(t)) \,, \\[4pt]
            x_i(0) = x_i \,,
        \end{cases}
        \quad \text{for } 1 \leq i \leq n \,,
    \end{equation}
    have a unique solution $(t \in [0,1] \mapsto x_i(t), 1 \leq i \leq n)$. Denote 
    \begin{equation}
        \mu_t \coloneqq \sum_{i=1}^n \frac{1}{n}\delta_{x_i(t)} \,.
    \end{equation}
    From Lemma \ref{lemma:time_dep_bound}, the support of $\mu_t$ is uniformly bounded in time.
    Using Proposition \ref{prop:tv_control_existence_homeo_bilip}, there exists a bi-Lipschitz homeomorphism curve $t \in [0,1] \mapsto \psi_t$ such that $\psi_t$ is the unique solution of the flow equation
    \begin{equation}
        \begin{cases}
            \partial_t \psi_t = v_t \circ \psi_t \, \\
            \psi_0 = \Id \,.
        \end{cases}
    \end{equation}
    Applying estimate \eqref{eq:estim_homeo_lip_tv} applied to $\psi_t^{-1}(x)$ and $\psi_t^{-1}(y)$ provides, for all $x \in \R^d$:
    \begin{equation}
        |\psi_t(x)-\psi_t(y)| \geq |x-y|\exp{\left( -\int_0^1 \|p\|_{L^2(\mu_t)} dt \right)} \geq |x-y|\exp{\left( -\int_0^1 \|p\|^2_{L^2(\mu_t)} dt \right)}\,.
    \end{equation}
    Furthermore, since the homeomorphism $\psi_t$ satisfies, at all times:
    \begin{equation}
        x_i(t) = \psi_t(x_i) \,,
    \end{equation}
    the points $(x_i(t), 1 \leq i \leq n)$ are distinct at all times, with a uniform bound from below of the minimal distance between points.
    For the last statement, remark that the ODE system \eqref{eq:lag_eq_particl} implies that, for all $1 \leq i \leq n$:
    \begin{equation}
        x_i(t) = x_i + \int_0^1 \sum_{j=1}^n p_{i,s}K(x_i(s)-x_j(s))\, ds\,,
    \end{equation}
    which is a weakly continuous application in $p$ for the $L^2([0,1])$ topology.
    
    For the converse part of the proposition, let $\mu_t \in AC^2([0,1],\mathcal{P}(\Omega))$. Then, from \cite[Theorem 8.3.1]{ambrosio2005gradient}, there exists a Borel velocity field $v_t \in L^2(\mu_t,\R^d)$ such that the continuity equation 
    \begin{equation}
        \partial_t \mu_t + \nabla \cdot (v_t \mu_t) = 0
    \end{equation}
    holds in the sense of distribution. This amounts to the existence, for almost all times $t \in [0,1]$, of a velocity $v_{t,i} \in \R^d$ such that 
    \begin{equation}
        \dot{x_i}(t) = v_{t,i} \,.
    \end{equation}
    Let us denote $X(t) = (x_1(t),\dots,x_n(t))$ and $K_{X(t)}$ the matrix with general entries $-|x_i(t) - x_j(t)|$.
    By Proposition \ref{prop:induced_norm}, for almost all $t$, there exists a unique couple $(p_t,\alpha_t) \in \left(\R^d \right)^n \times \R^d$ such that $\sum_i p_{t,i} = 0_d$ and such that, for all $1 \leq i \leq n$:
    \begin{equation}
        v_{t,i} = (K_{X(t)}p_t)_i + \alpha_t \,.
    \end{equation}
    Define 
    \begin{equation}
        \Gamma_t(p) = \sum_{i=1}^n 
        p_{t,i}\delta_{x_i(t)} \,.
    \end{equation}
    Since the trajectories $x_i(t)$ are distinct at all times, there exists $r,\sigma > 0$ such that $\Gamma_t \in \mathcal{M}_{\mathcal{P}_{m-1},n,r}(B_{\sigma r})$ at all times. This proves, using estimate \eqref{eq:control_L2_norm_rkhs} from Proposition \ref{prop:tv_maj_if_no_collpase} with the same $\lambda_\sigma >0$, that:
    \begin{equation}
        |p_t|_2^2 \leq \frac{\lambda_\sigma^2}{r} |K_{X(t)} p_t|_2^2 \,.
    \end{equation}
    As $\int_0^1 |K_{X(t)} p_t|^2dt = \int_0^1 |v_t|^2d\mu_t < +\infty$, this proves that $t \mapsto p_t$ is in $L^2\left([0,1],\left( \R^d \right)^n\right)$, concluding the proof.
\end{proof}
This property allows for a reparametrization of the matching problem \eqref{eq:prob_reg_L2_mom} by the set $L^2\left([0,1],\left( \R^d \right)^n\right)$ in the case of $n$ landmarks. We consider the following problem:
\begin{equation}\label{eq:prob_reg_L2_mom_extr}
\underset{ p \in L^2\left([0,1],\left( \R^d \right)^n\right)}{\inf}
    \Tilde{E}^\lambda(p) = \mathcal{L}(\mu_1,\nu) + \int_0^1 \langle \Gamma_t, K \star \Gamma_t \rangle dt + \int_0^1 |p_t|^2dt \,,
\end{equation}
where
\begin{equation}
    \mu_0 = \sum_{i=1}^n \frac{1}{n} \delta_{x_i}\,,
\end{equation}
the $x_i$ are all distinct, $\mu_t$ is the (unique) absolutely continuous curve provided by Proposition \ref{prop:L2_mom_reparam},
\begin{equation}
    \Gamma_t = \sum_{i=1}^n p_{t,i} \delta_{x_i(t)} \,,
\end{equation}
and $\mathcal{L}$ is a general loss.
We are able to prove the following result:
\begin{theorem}\label{th:exist_minim_parti_L2_regu}
    Suppose that $\mathcal{L}$ is a lower semicontinuous loss. Then, there exist solutions to the variational problem \eqref{eq:prob_reg_L2_mom_extr}.
\end{theorem}
\begin{proof}
    The level sets $\left\{ p  \in L^2\left([0,1],\left( \R^d \right)^n\right) \,|\, \int_0^1 |p_t|^2 dt \leq C \right\}$ are weakly compact, and $p \mapsto \Tilde{E}^\lambda(p)$ is weakly l.s.c. by Proposition \ref{prop:L2_mom_reparam} and the hypothesis on $\mathcal{L}$. This proves the result.
\end{proof}
We can prove the existence of minimizers to Problem \eqref{eq:prob_reg_L2_mom} for very general initial measures:
\begin{theorem}
     Let $\mathcal{L}$ be a lower semicontninuous loss. Then, there exists a solution to the variational problem \eqref{eq:prob_reg_L2_mom}.
\end{theorem}
\begin{proof}
    First, note that if $\gamma \in L^2([0,1],L^2(\mu_t))$, then the moment measure $\Gamma_t \coloneqq \gamma_t \mu_t$ is well defined for almost all times, and satisfies:
    \begin{equation}
        E^\lambda(\mu,\gamma) = \Tilde{E}^\lambda (\mu,\Gamma) \coloneqq \mathcal{L}(\mu_1,\nu) + \int_0^1 \langle \Gamma_t, K \star \Gamma_t \rangle dt + \int_0^1 \int \left|\frac{d \Gamma_t}{d\mu_t}\right|^2d\mu_t dt\,.
    \end{equation}
    Let $(\mu^k,\gamma^k)$ be a minimizing sequence to problem \eqref{eq:prob_reg_L2_mom}. Note that there exists a uniform constant $C >0$ such that, for all $k \geq 0$:
    \begin{equation}
        \int_0^1 |\gamma_t^k|^2_{L^2(\mu_t^k)}dt \leq C \,.
    \end{equation}
    By Lemma \ref{lemma:time_dep_bound}, the support of $\mu_t^k$ is uniformly bounded in time and in $k$. This proves that $\mu^k \in AC^2([0,1],\mathcal{P}(\Omega))$ for some compact set $\Omega$, so that there exists an accumulation point $\mu \in AC^2([0,1],\mathcal{P}(\Omega))$. Now, let us define, for any $k \geq 0$, the momentum measure:
    \begin{equation}
        \Gamma_t^k \coloneqq \gamma_t^k \mu_t^k \,.
    \end{equation}
    Since, for any $k$:
    \begin{equation}
        \int_0^1 |\Gamma_t^k|_{TV}dt \leq C \,,
    \end{equation}
    there exists a subsequence such that $\Gamma^k$ weakly converges to an accumulation point $\Gamma \in \mathcal{M}([0,1];\R^d)$, in the sense that for any test function $f \in \mathcal{C}^0([0,1]\times\Omega)$, 
    \begin{equation}
        \int_0^1 \int_\Omega f_t d\Gamma_t^k \underset{k \rightarrow + \infty}{\rightarrow}   \int_0^1 \int_\Omega f_t d\Gamma_t \,.
    \end{equation}
    Now, define the following velocity field:
    \begin{equation}
        v_t^k \coloneqq K \star \Gamma_t^k \,,
    \end{equation}
    which is defined for almost any $t$ on $\Omega$.
    The Energy-Distance kernel $K$ is continuous, so that if $g  \in L^2([0,1],L^2(\mu_t))$ is a test function:
    \begin{align}
        \int_0^1 \int g_t v_t^k d\mu_tdt &= \int_0^1 \int g_t(x)\left(\int K(x-y)d\Gamma_t^k(y)\right)d\mu_t(x)dt\\
        &\underset{k \rightarrow \infty}{\rightarrow} \int_0^1 \int g_t(x)\left(\int K(x-y)d\Gamma_t(y)\right)d\mu_t(x)dt \\ 
        &= \int_0^1 \int g_t v_t d\mu_tdt \,.
    \end{align}
    This proves that, for almost all times $t$, the sequence $v_t^k$ converges to $v_t \coloneqq K \star \Gamma_t$ in $L^2(\mu_t)$. Moreover, the continuity equation is weakly verified at the limit. Indeed, for any $\varphi \in \mathcal{C}^\infty([0,1]\times \Omega)$:
    \begin{equation}
        0 = \int_0^1 \left( \partial_t \varphi + v_t^k\cdot\nabla \varphi_t \right)d\mu_t^ndt \underset{k}{\rightarrow }\int_0^1 \left( \partial_t \varphi + v_t^k\cdot\nabla \varphi_t \right)d\mu_t^ndt\,.
    \end{equation}
    Now, since the Lagrangian
    \begin{equation}
        (\mu,\Gamma) \mapsto \int_0^1 \langle \Gamma_t, K \star \Gamma_t \rangle dt + \int_0^1 \int \left|  \frac{d\Gamma_t}{d\mu_t}\right|^2d\mu_t
    \end{equation}
    is convex, it is lower semi-continuous, so that by hypothesis on $\mathcal{L}$ the functional $\Tilde{E}^\lambda$ is lower semi-continuous. This proves:
    \begin{equation}
        \Tilde{E}^\lambda(\mu,\Gamma) \leq \underset{(\mu,\gamma) \text{ solves } \eqref{eq:sys_transp_momentum} }{\inf}
    E^\lambda(\mu,\gamma)\,.
    \end{equation}
    Finally, let us define 
    \begin{equation}
        \gamma_t = \frac{d \Gamma_t}{d\mu_t}\,,
    \end{equation}
    the Radon-Nikodym derivative of $\Gamma_t$ by $\mu_t$. Then $\gamma \in L^2([0,1],L^2(\mu_t))$, and $\Tilde{E}^\lambda(\mu,\Gamma) = E^\lambda(\mu,\gamma)$, ending the proof.
\end{proof}
This solution has additional structure: at all times, the measure $\mu_t$ is the pushforward of $\mu_0$ through a bi-Lipschitz homeomorphism.
\begin{proposition}
    Let $(\mu,\gamma)$ be a solution to the system \eqref{eq:sys_transp_momentum} with finite $E^\lambda$ energy. Then, there exists a bi-Lipschitz homeomorphism curve $\psi_t$ such that $\psi_t$ is the unique solution to the flow equation
    \begin{equation}
        \begin{cases}
            \partial_t \psi_t = v_t \circ \psi_t \, \\
            \psi_0 = \Id \,.
        \end{cases}
    \end{equation}
    Furthermore, at all times $t \in [0,1]$, we have: $\mu_t = \psi_{t\#} \mu_0$. 
\end{proposition}
\begin{proof}
    This a direct consequence of Proposition \ref{prop:tv_control_existence_homeo_bilip} and of the inequality $|\Gamma|_{TV} \leq \|\gamma\|_{L^2(\mu)}$.
\end{proof}




%% file: Sliced_ED.tex
\section{Sliced Energy Distance.}\label{SecSlicedED}
In this section, we show how various interaction functionals involving the Energy-Distance can be statistically estimated in $O(n\log n)$ time, where $n$ is the number of points. This almost-linear complexity is a significant improvement over the $O(n^2)$ time of more naive methods, making it suitable for large-scale applications.
The approach follows the same principle as the Sliced Wasserstein distance \cite{julien:hal-00476064,bonneel:hal-00881872,nadjahiphdthesis}: the original metric is replaced with a sliced approximation over 1D distributions. Applications of this principle to the Energy-Distance kernel build on ideas from \cite{hertrich2024generativeslicedmmdflows}, where an efficient method for computing the gradient of the Energy-Distance loss between two sets of $n$ Dirac masses was established. We extend these results, computing general convolutions of the kernel matrix with 1D atomic distributions.  Specifically, the convolution of the Energy-Distance kernel with any 1D measure composed of $n$ Dirac masses is computable in quasi-linear time using a sorting procedure, leveraging particularly efficient computations when the 1D points are sorted. We prove that the convolution of atomic measures in higher dimensions can be efficiently estimated using a tractable number of projections of the measures over 1D lines, with all the computations being easily parallelizable. 
\subsection{Fast estimators.}
First, we establish that, in the case of sorted 1D data, the convolution can be computed in linear time.
\begin{proposition}[Linear time complexity for sorted data]\label{prop:ed_calc_dim1}
    Let $x_0 \leq x_1 \leq \dots \leq  x_N \in \R$ and $\gamma_0,\dots,\gamma_N \in \R$. Then, the quantity
    \begin{equation}
        K \star \gamma \coloneqq \left( \sum_{j=0}^n \gamma_j |x_i - x_j| ; i \in [1,n] \right) \in \R^n
    \end{equation}
    can be computed in $O(n)$ time complexity.
\end{proposition}

\begin{proof}
    Let $i \in [0,n]$. We rewrite the sum as follows:
    \begin{align}
        \sum_{j=0}^n \gamma_j |x_i - x_j| &= \sum_{j=0}^{i-1} \gamma_j (x_i - x_j) + \sum_{j=i+1}^{n} \gamma_j (x_j - x_i) \\
        &= x_i \left( \sum_{j=0}^{i-1} \gamma_j -  \sum_{j=i+1}^{n} \gamma_j \right) - \left( \sum_{j=0}^{i-1} \gamma_j x_j - \sum_{j=i+1}^{n} \gamma_j x_j \right) \\
        \sum_{j=0}^n \gamma_j |x_i - x_j| &= a_i x_i - b_i \,,
    \end{align}
    where $a_i$ and $b_i$ are scalar terms that satisfy the recurrence relations:
    \begin{equation}
    \begin{aligned}
        &\begin{cases}
            a_0 &= - \sum_{j=1}^n \gamma_j \\
            a_{i+1} &= a_i + \gamma_i + \gamma_{i+1}
        \end{cases}
        &&\text{and} \quad
        \begin{cases}
            b_0 &= - \sum_{j=1}^n \gamma_j x_j \\
            b_{i+1} &= b_i + \gamma_i x_i + \gamma_{i+1} x_{i+1}
        \end{cases}.
    \end{aligned}
\end{equation}
    Computing $((a_i,b_i);i\in[0,n])$ requires $O(n)$ operations, which concludes the proof.
\end{proof} 
The convolution of the Energy-Distance kernel with any finite measure in $\R^d$ is equal to its sliced version.
\begin{proposition}[Sliced Energy-Distance]\label{prop:sliced_ed_convol}
    Let $\gamma \in \mathcal{M}(\R^d)$ be a finite signed measure. For $\theta \in \mathcal{S}_{d-1}$, define the 1D projection on the line $\R \theta$ via $\pi_{\theta}(x) \coloneqq \langle x,\theta \rangle \theta $. Then, the convolution with the Energy-Distance kernel is sliced, meaning that there exists a positive constant $c_d$  such that, for all $x \in \R^d$:
    \begin{equation}
        K \star \gamma(x) = c_d \int_{\mathcal{S}_{d-1}} K \star (\pi_{\theta\#} \gamma)(x) d\sigma(\theta) \,,
    \end{equation}
    and
    \begin{equation}
        \langle \gamma , K \star \gamma \rangle = c_d \int_{\mathcal{S}_{d-1}} \langle \pi_{\theta \#}\gamma , K \star (\pi_{\theta \#} \gamma) \rangle d\sigma(\theta) \,,
    \end{equation}
    where $\sigma$ is the standard uniform probability measure on the sphere $\mathcal{S}_{d-1}$, and 
    \begin{equation}
        c_d = \frac{\sqrt{\pi}\Gamma\left( \frac{d+1}{2} \right)}{\Gamma\left( \frac{d}{2} \right)}\,.
    \end{equation}
\end{proposition}
\begin{proof}
    The proof follows directly from adapting the result of \cite[Theorem 1]{hertrich2024generativeslicedmmdflows} to a general finite measure $\gamma$.
\end{proof}
We can obtain a bound on the error resulting from approximating the integral above through a Monte-Carlo method.
\begin{proposition}[Statistical error]\label{prop:sliced_ed_convol_statistical_error}
    Let $\gamma \in \mathcal{M}(\R^d)$ be a finite signed measure with support in a ball of radius $R$, and let $x \in \R^d$. Let $\theta_1, \cdots, \theta_P$ be $P$  i.i.d. random variables uniformly distributed on the sphere $\mathcal{S}_{d-1}$. Then, the following estimators
    \begin{equation}
        \widehat{K \star \gamma(x)} \coloneqq \frac{c_d}{P} \sum_{p=1}^P K \star (\pi_{\theta_p \#}\gamma)(x) \,,
    \end{equation}
    and 
    \begin{equation}
        \widehat{\langle \gamma , K \star \gamma \rangle} \coloneqq \frac{c_d}{P} \sum_{p=1}^P \langle \pi_{\theta_p \#}\gamma , K \star (\pi_{\theta_p \#} \gamma) \rangle \,
    \end{equation}
    satisfy
    \begin{equation}\label{eq:convol_kernel_eval_estim}
        \mathbb{E}_{\theta_1, \cdots, \theta_P \sim \mathcal{U}(\mathcal{S}_{d-1})}(\|\widehat{K \star \gamma(x)} - K \star \gamma(x) \| ) \leq AR|\gamma|\sqrt{\frac{d}{P}}\,,
    \end{equation}
    and
    \begin{equation}\label{eq:scalar_kernel_prod_estim}
        \mathbb{E}_{\theta_1, \cdots, \theta_P \sim \mathcal{U}(\mathcal{S}_{d-1})}(|\widehat{\langle \gamma , K \star \gamma \rangle} - \langle \gamma , K \star \gamma \rangle | ) \leq AR|\gamma|^2\sqrt{\frac{d}{P}}\,.
    \end{equation}
    where $|\gamma|$ denotes the total variation of the measure $\gamma$, for some constant $A>0$.
\end{proposition}
\begin{proof}
    In this proof we write $\mathbb{E}$ and $\mathbb{P}$ instead of $\mathbb{E}_{\theta_1, \cdots, \theta_P \sim \mathcal{U}(\mathcal{S}_{d-1})}$ and $\mathbb{P}_{\theta_1, \cdots, \theta_P \sim \mathcal{U}(\mathcal{S}_{d-1})}$.
    Let $x \in \R^d$.
    First, remark that if $\theta_1, \cdots, \theta_P \sim \mathcal{U}(\mathcal{S}_{d-1})$ then almost surely:
    \begin{equation}
        \|\widehat{K \star \gamma(x)} - K \star \gamma(x)\| \leq 2 c_d R|\gamma|\,.
    \end{equation}
    This directly implies
    \begin{equation}
        \mathbb{E}(\|\widehat{K \star \gamma(x)} - K \star \gamma(x)\|^2) \leq 4 (c_d R|\gamma|)^2\,.
    \end{equation}
    Now, we can apply Bernstein's inequality to the zero-mean variables $\widehat{K \star \gamma(x)} - K \star \gamma(x)$, to get:
    \begin{equation}
        \mathbb{P}(\|\widehat{K \star \gamma(x)} - K \star \gamma(x)\| \geq t) \leq 2\exp\left( \frac{-\frac{1}{2}P^2t^2}{4P(c_d R|\gamma|)^2 + \frac{1}{3}|\gamma|c_d R P t} \right) = 2\exp\left( -\frac{3P}{4 c_d R|\gamma|}\frac{t^2}{6 c_d R|\gamma|+t} \right) \,.
    \end{equation}
    For any $\alpha > 0$, integrating the inequality over $t$:
    Then:
    \begin{align}
        \int_0^\alpha \mathbb{P}(\|\widehat{K \star \gamma(x)} - K \star \gamma(x)\| \geq t) dt &\leq \int_0^\alpha 2\exp\left( -\frac{3P}{4 c_d R|\gamma|}\frac{t^2}{6 c_d R|\gamma|+\alpha} \right)dt \\
        &\leq 2\sqrt{\frac{\pi c_d R |\gamma| (6 c_d R|\gamma|+\alpha) }{3P}}\,.
    \end{align}
    Furthermore, we can get the bound:
    \begin{align}
        \int_\alpha^\infty \mathbb{P}(\|\widehat{K \star \gamma(x)} - K \star \gamma(x)\| \geq t) dt &\leq \int_\alpha^\infty 2\exp\left( -\frac{3P}{4 c_d R|\gamma|}\frac{t^2}{6 c_d R|\gamma|+t} \right)dt \\
        &\leq \int_\alpha^\infty 2\exp\left( -\frac{3Pt}{4 c_d R|\gamma|}\left(1-\frac{6 c_d R|\gamma|}{t}\right) \right)dt \\
        &=2\exp\left( \frac{9P}{2} \right)\frac{4 c_d R|\gamma|}{3P}\exp\left( -\frac{3P\alpha}{4 c_d R|\gamma|} \right)\,.
    \end{align}
    Choosing $\alpha = 7 c_d R|\gamma|$ for example ensures the existence of a constant $A$ such that:
    \begin{equation}
        \mathbb{E}(\|\widehat{K \star \gamma(x)} - K \star \gamma(x)\|) \leq AR|\gamma|\frac{c_d}{\sqrt{P}}\,.
    \end{equation}
    As $\frac{\Gamma(\frac{d+1}{2})}{\Gamma(\frac{d}{2})}$ behaves asymptotically as $O(\sqrt{d})$, we get the estimate \eqref{eq:convol_kernel_eval_estim}. The second estimate \eqref{eq:scalar_kernel_prod_estim} is obtained by integrating the first estimate over $\gamma$.
\end{proof}
Below is an illustration of the previous results. We consider an atomic measure $\mu$ with support on a set of distinct points $X = (x_1, \dots,x_N) \in (\R^d)^N$, with associated moment $\gamma_i$ in $\R^d$ at each point. We compute the vector $(K\star \gamma\mu(x_i),1\leq i N)$, and compare it with its sliced approximations.
\begin{figure}[h!]
    \centering
    \includegraphics[width=0.3\linewidth]{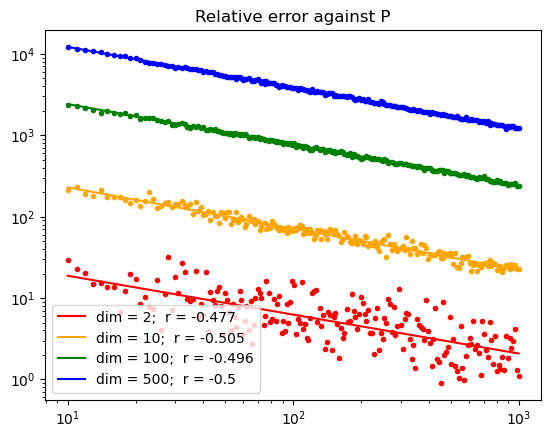}
    \includegraphics[width=0.3\linewidth]{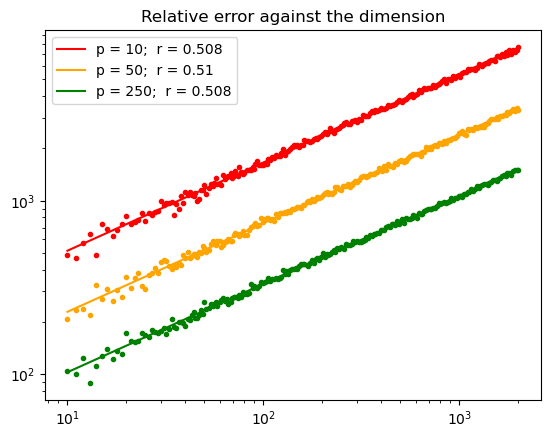}
    \includegraphics[width=0.3\linewidth]{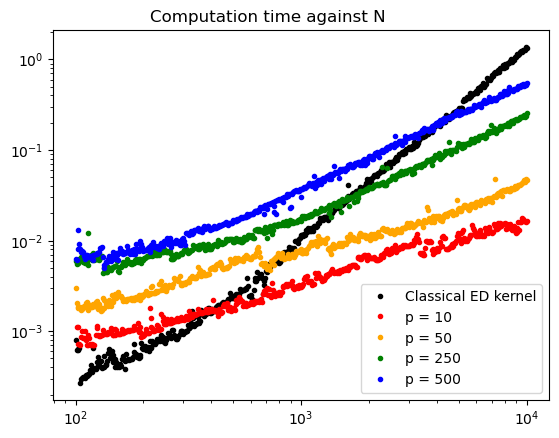}
    \caption{In the left figure, we plot the error convergence to zero in $P^{-1/2}$ between the approximated vector and the ground-truth as $P$ grows. In the middle, the error as the dimension increases. In the right-most graph, we plot the computing time against the number of points in dimension 5 on a standard 8-core CPU.}
    \label{fig:enter-label}
\end{figure}
\subsection{Use as a loss.}
Computing the Energy-Distance loss between two points clounds $(x_1,\cdots,x_N) \in \left(\R^d\right)^N$ and $(y_1,\cdots,y_M) \in \left(\R^d\right)^M$ amounts to computing $\langle \rho_{X,Y},K\star\rho_{X,Y}\rangle$, where
\begin{equation}
    \rho_{X,Y} \coloneqq \sum_{i=1}^N \frac{1}{N}\delta_{x_i} - \sum_{j=1}^M \frac{1}{M}\delta_{y_j}\,.
\end{equation}
    The previous discussion shows that $\langle \rho_{X,Y},K\star\rho_{X,Y}\rangle$ can be estimated with an $O((N+M)\log(N+M))$ complexity.
    

%% file: experiments.tex
\section{Experiments.}
\label{SecExperiments}

To demonstrate the matching properties of the Energy-Distance kernel, we implement and evaluate the solution to Problem \eqref{eq:min_prob_without_tv} across multiple point-cloud datasets. The trajectories are parametrized by their starting points and moment values at all times (see \ref{prop:L2_mom_reparam}), and computed using a simple Euler scheme. We have to account for the constraints needed in the case of the Energy-Distance kernel, and parameterize the null-space of the metric explicitly. Let us define the projection matrix $\Pi_n \coloneqq \Id_n - \frac{1}{n}\mathbbm{1}_n$, where $\mathbbm{1}_n$ is the $n \times n$ matrix with $1$ entries. If $P_t = (P_{t,l}, 1\leq l \leq d)  \in \R^{n \times d}$, the projection on the constraint set we choose is defined by:
\begin{equation}
    \overline{P_t} \coloneqq (\Pi_n P_{t,l}, 1\leq l \leq d)  \in \R^{n \times d}\,.
\end{equation}We detail the computations in the algorithm below (recall that $K_{X}$ denotes the matrix with general entries $K(x_i,x_j)$ for $X = (x_1,\dots,x_n)$):
\begin{algorithm}[H]
    \caption{Regularized trajectories for a kernel $K$ (Euler scheme).}
    \label{algo:euler_trajectories}
    \begin{algorithmic}[1]
        \Require {Initial data $X_0 = (x_1, \dots, x_{n}) \in \mathbb{R}^d$, $T_{discr}$ the number of time discretizations, moment vector $P = (P_t, 0 \leq t < T_{discr}) \in \R^{T_{discr}\times n \times d}$. If $K$ is the Energy-Distance kernel, $(\alpha_t,0\leq t < T_{discr}) \in \R^{T_{discr}\times d}$ is a parametrization of translations.}
        \State $X_{traj}$
        \For {$t=0, \dots, T_{discr}-1$}
            \State Compute the velocity field
            \begin{equation}\label{algo:v_update}
                v_{t,i} \gets \begin{cases}
                        (K_{X_t}\overline{P_t})_i + \alpha_t &\text{if }$K$ \text{ is the ED kernel,}   \\
                        (K_{X_t}P_t)_i  & \text{otherwise.} 
                    \end{cases}
            \end{equation}
            \State Compute the positions at time $t+1$:
            \begin{equation}
                x_i(t+1) \gets x_i(t) + \frac{1}{T_{discr}}v_t(x_i(t),P_t)
            \end{equation}
        \EndFor
        \State \Return $X(P) = (X_t(P), 0\leq t \leq T_{discr}) \in \R^{(T_{discr}+1)  \times n \times d}$.
    \end{algorithmic}
\end{algorithm}
Note that the computations needed for $v_t$ \eqref{algo:v_update} have an $O(n^2)$ time-complexity for general kernels, but $O(n\log{n})$ time-complexity for the sliced computations of the Energy-Distance kernel. 
We propose to solve the constrained optimization problem \eqref{eq:min_prob_without_tv} using an Augmented Lagrangian method \cite{bertsekas2014constrained} with adaptive penalty update with an inner LBFGS optimization. The algorithm operates with a kernel $K$ as follows:

\begin{algorithm}[H]
    \caption{Augmented Lagrangian registration.}
    \begin{algorithmic}[1]
        \Require {Initial data $X_0 = (x_1, \dots, x_{n}) \in \mathbb{R}^d$, target data $Y_0 = (y_1, \dots, y_{m}) \in \mathbb{R}^d$, $\mathcal{L}(\cdot,\cdot)$ the loss between point-clouds, $T_{discr}$ the number of time discretizations, $k_{max}$ the number of optimization steps, $\varepsilon$ the tolerance to the loss constraint, $\rho_{init}$ a parameter.}
        \State $P \gets 0 \in \mathbb{R}^{T_{discr}\times n \times d   } \,.$
        \State $\lambda \gets 0 \,.$
        \State $\rho \gets \rho_{init}\,.$
        \For {$k=1, \dots, k_{max}$}
            \State Define the augmented Lagrangian 
            \begin{equation}
                \mathcal{L}_{\lambda,\rho}(X(P),P) \coloneqq \frac{\rho}{2}\mathcal{L}(X_1(P),Y_0)^2 + \lambda\mathcal{L}(X_1(P),Y_0) + \frac{1}{T_{discr}}\sum_{t=1}^{T_{discr}} \langle P_t , K_{X_t(P)} P_t \rangle \,.
            \end{equation}
            \State Compute the deformation curve $X(P) = (X_t(P),1\leq t \leq T_{discr})$ using Algorithm \ref{algo:euler_trajectories}. Note that the third term of the augmented Lagrangian can be computed at the same time.
            \State Update $P$ via an LBFGS optimization step of the augmented Lagrangian $\mathcal{L}_{\lambda,\rho}(X(P),P) \,.$
            \State Update $\lambda$ and $\rho$ :
            \begin{equation}
                \begin{aligned}
                    &\lambda \gets \lambda + \rho\mathcal{L}(X_1(P),Y_0)
            &&\text{and} \quad
                    &\rho \gets
                    \begin{cases}
                        1.2\rho &\text{   if   } \mathcal{L}(X_1(P),Y_0) > \varepsilon \,,  \\
                        \rho & \text{otherwise.} 
                    \end{cases}
                \end{aligned}
            \end{equation}
        \EndFor
        \State \Return $X(P),P\,.$
    \end{algorithmic}
\end{algorithm}
When implemented with the sliced approximation of the augmented Lagrangian terms, the final algorithm achieves $O(n\log n)$ time complexity.
In the figure below, we apply this algorithm with a tuned $\sigma$ Gaussian kernel and with the Energy-Distance kernel. The loss we choose is a sliced Energy-Distance loss, with $P = 1000$ projections.

\begin{figure}[h!]
    \centering
    \begin{tabular}{cc}
    \includegraphics[width=0.3\linewidth]{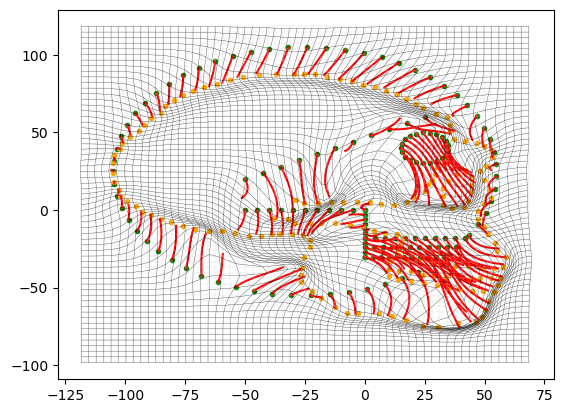}
    \includegraphics[width=0.3\linewidth]{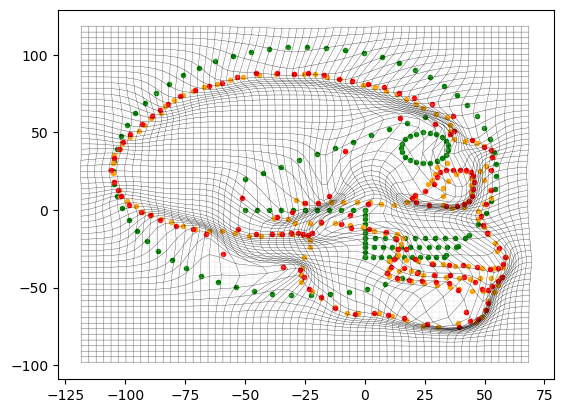}
    \includegraphics[width=0.3\linewidth]{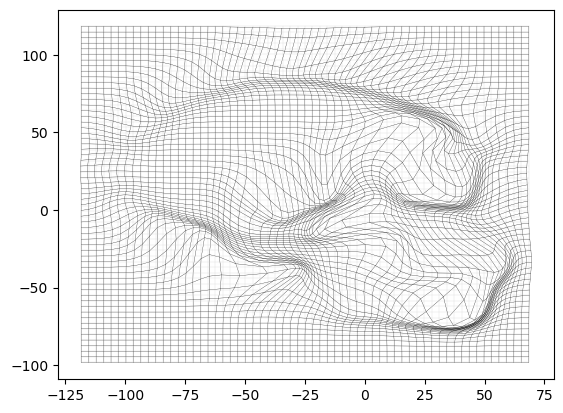} \\
    \includegraphics[width=0.3\linewidth]{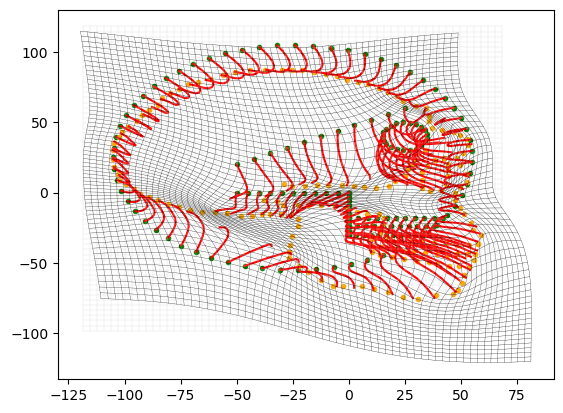}
    \includegraphics[width=0.3\linewidth]{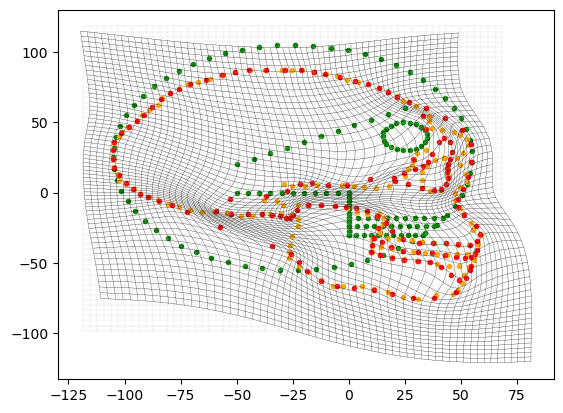}
    \includegraphics[width=0.3\linewidth]{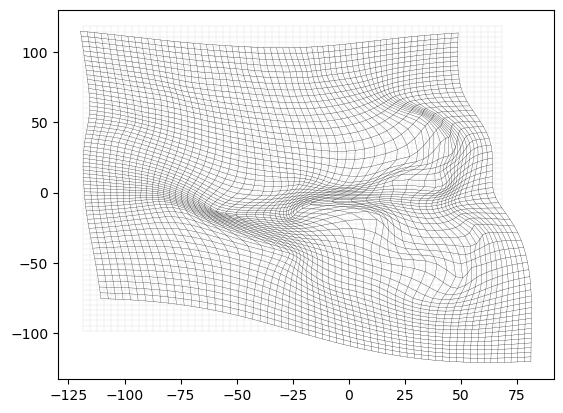}
    \end{tabular}
    \caption{Registration with Gaussian kernel at the top, Energy-Distance at the bottom. The final losses are comparable.}
    \label{fig:gaussian_vs_ed}
\end{figure}
In the case of the Energy-Distance kernel, we observe that the resulting deformation exhibits significantly less localized behavior when compared to the Gaussian kernel.
Moreover, the algorithm above can be applied with the sliced ED approximation of the particle path. The behavior seems highly stable, matching points with as low as $2$ projections, and a convergence to the true deformation as the number of projections grows.

\begin{figure}[h!]
    \centering
    \begin{tabular}{cccc}
    \includegraphics[width=0.24\linewidth]{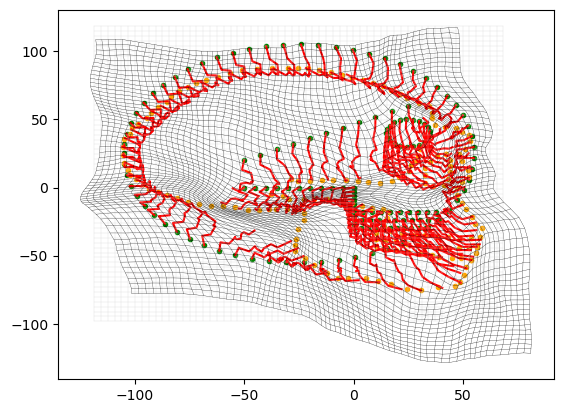}
    \includegraphics[width=0.24\linewidth]{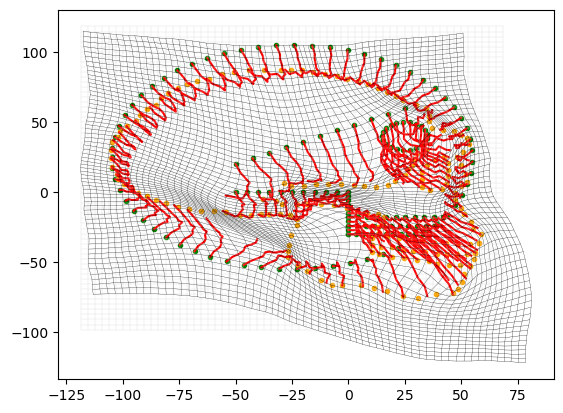}
    \includegraphics[width=0.24\linewidth]{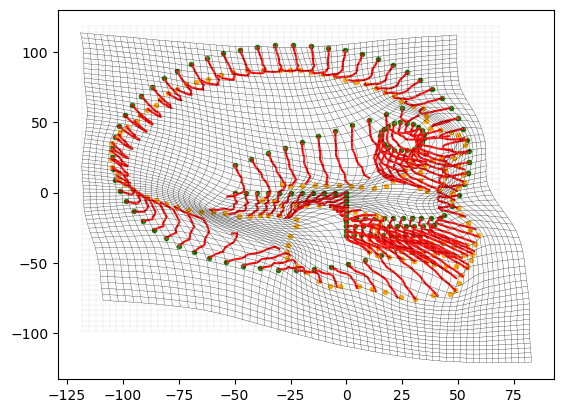}
    \includegraphics[width=0.24\linewidth]{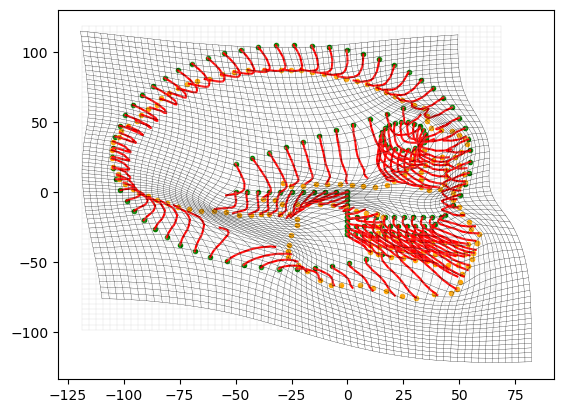} \\
    \includegraphics[width=0.24\linewidth]{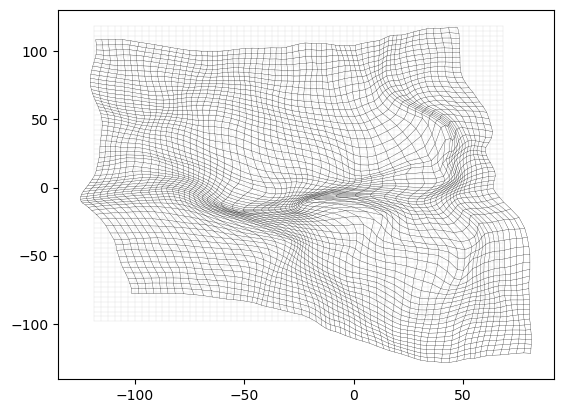} 
    \includegraphics[width=0.24\linewidth]{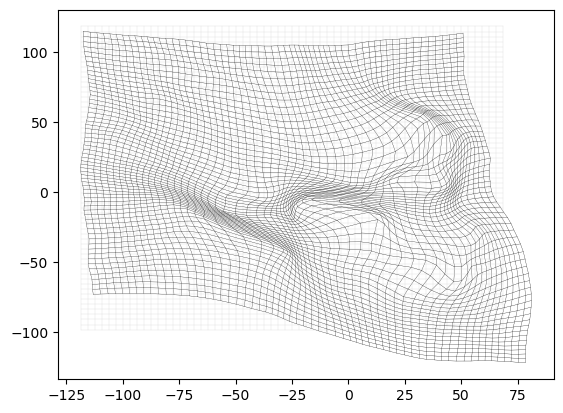} 
    
    \includegraphics[width=0.24\linewidth]{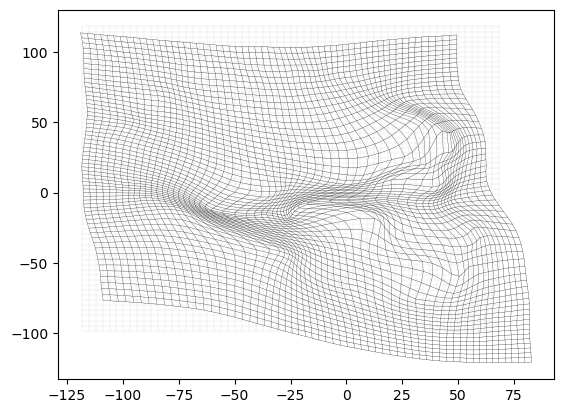} 
    
    \includegraphics[width=0.24\linewidth]{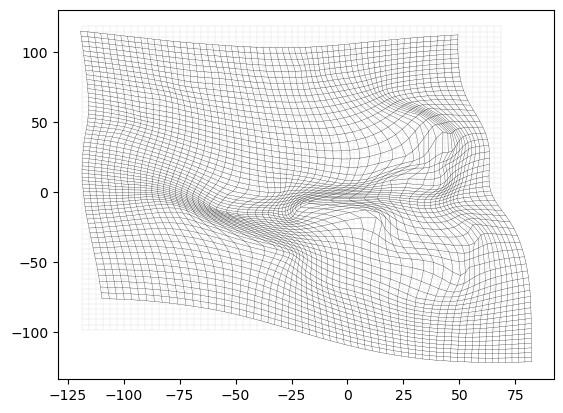}  
    \end{tabular}
    \caption{Same registration problem using the sliced computation of the kernel. From left to right $P = 2, 10, 50, 500\,.$}
    \label{fig:slicing_effect}
    
\end{figure}

The algorithm successfully recovers accurate matches even for small projection numbers $P$. Furthermore, the resulting deformations stabilize as $P$ increases, exhibiting convergence toward the true Energy-Distance (ED) registration. Notably, while the deformation paths may become less regular with fewer projections, the final registration quality remains comparable, suggesting robustness to the choice of $P$.
Moreover, Energy-Distance registration naturally handles translated shapes, as translations lie in the null space of the deformation metric.
\begin{figure}[h!]
    \centering
    \begin{tabular}{cccc}
        \includegraphics[width=0.24\linewidth]{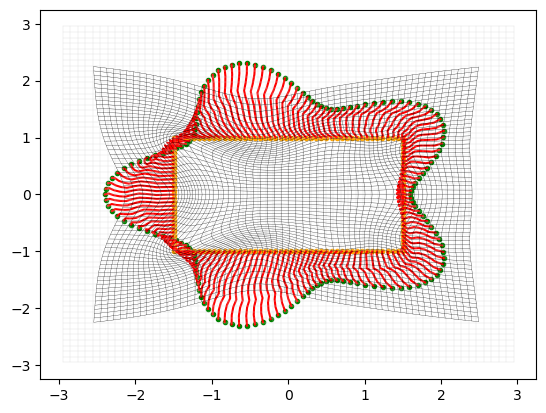}
        \includegraphics[width=0.24\linewidth]{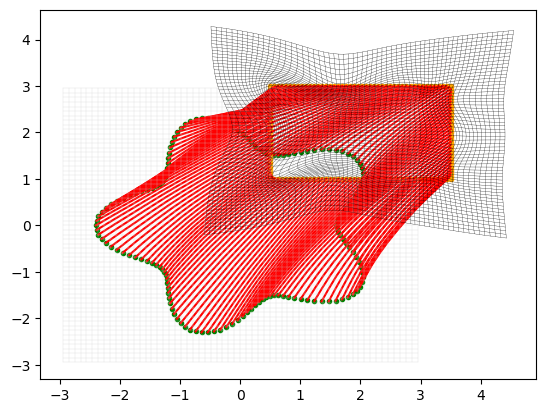}

        \includegraphics[width=0.24\linewidth]{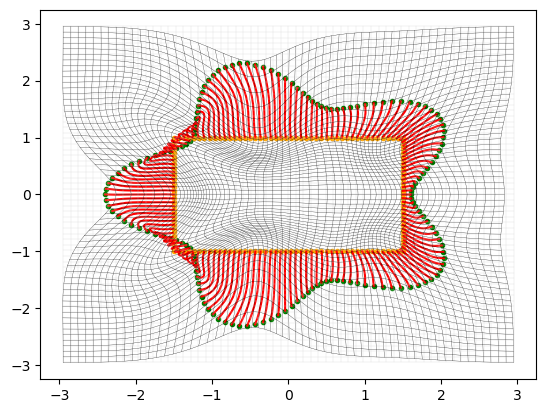}
        \includegraphics[width=0.24\linewidth]{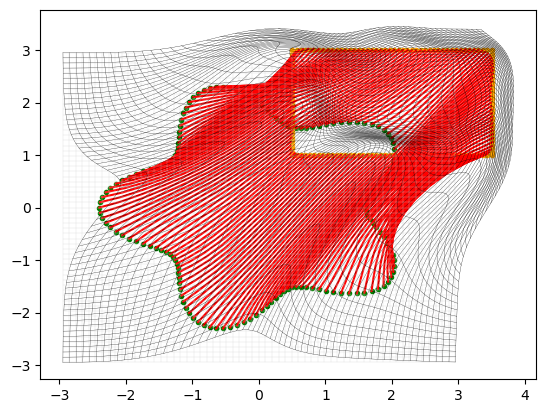} \\
        \includegraphics[width=0.24\linewidth]{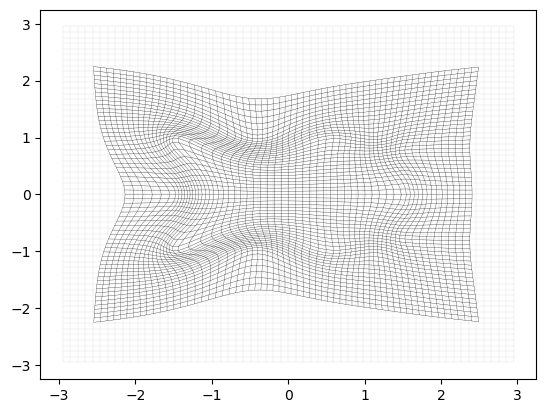}
        \includegraphics[width=0.24\linewidth]{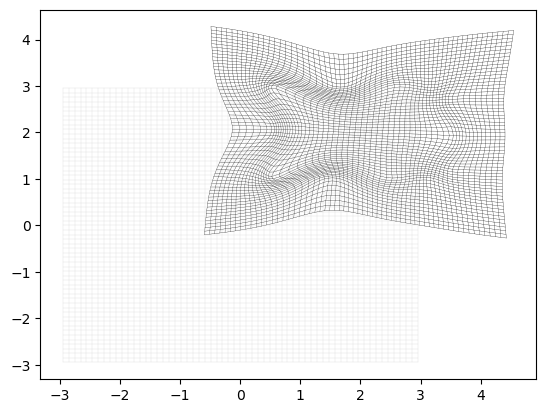}
        \includegraphics[width=0.24\linewidth]{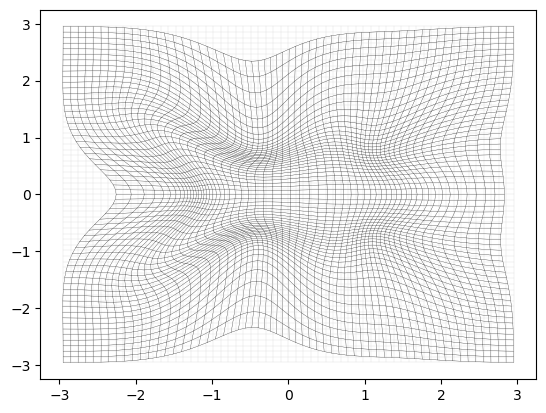}
        \includegraphics[width=0.24\linewidth]{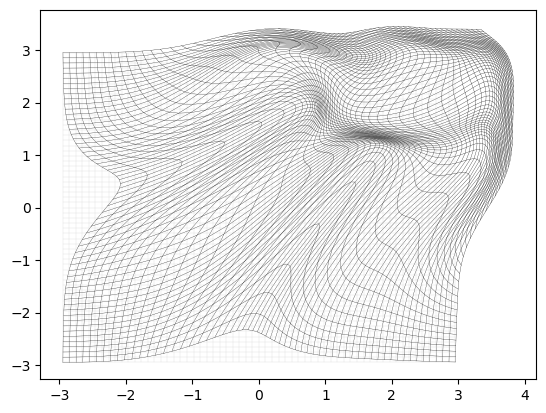}
    \end{tabular}
    \caption{Visualization of translated registration results: Energy-Distance kernel (left) and a tuned Gaussian kernel results (right).}
    \label{fig:translated_registration}
\end{figure}
Our method with sliced computations scales to a large number of points. The matching accuracy is inherently limited by random sampling, which does not uniformly weight all regions of interest of the shapes.  Moreover, the current method does not incorporate additional geometric or contextual information other than the point cloud coordinates. The following experiment can be run in under a minute on a standard recent computer. 

\begin{figure}[h!]
    \centering
    \begin{tabular}{cc}
    \includegraphics[width=0.25\linewidth]{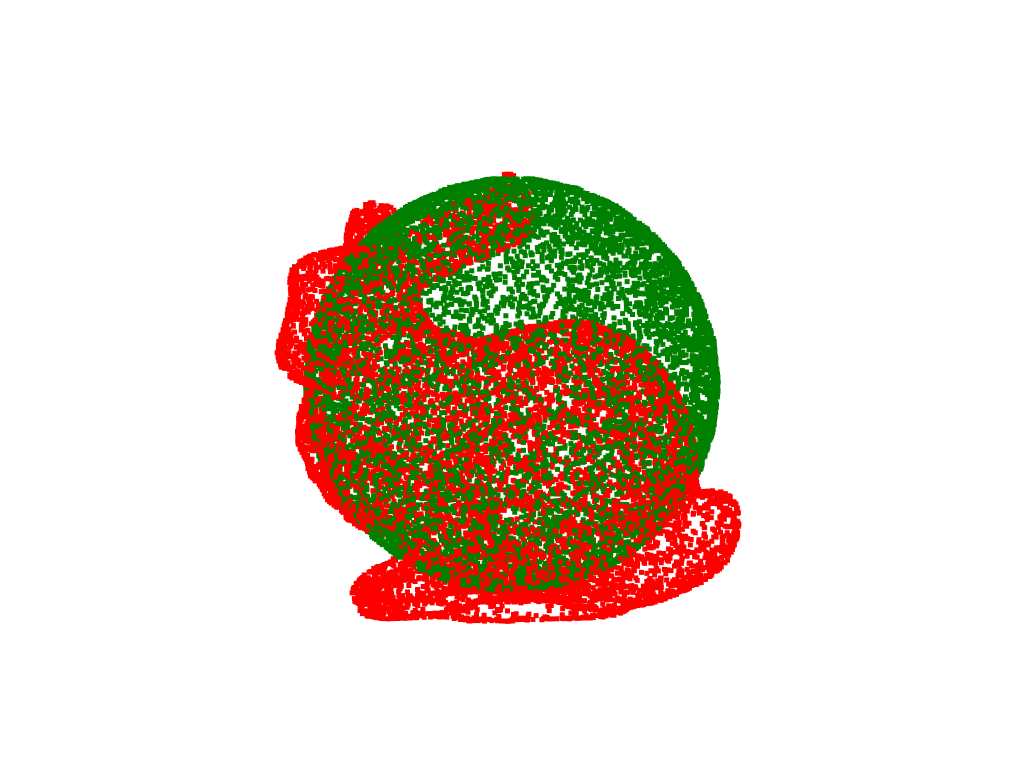}
    \includegraphics[width=0.25\linewidth]{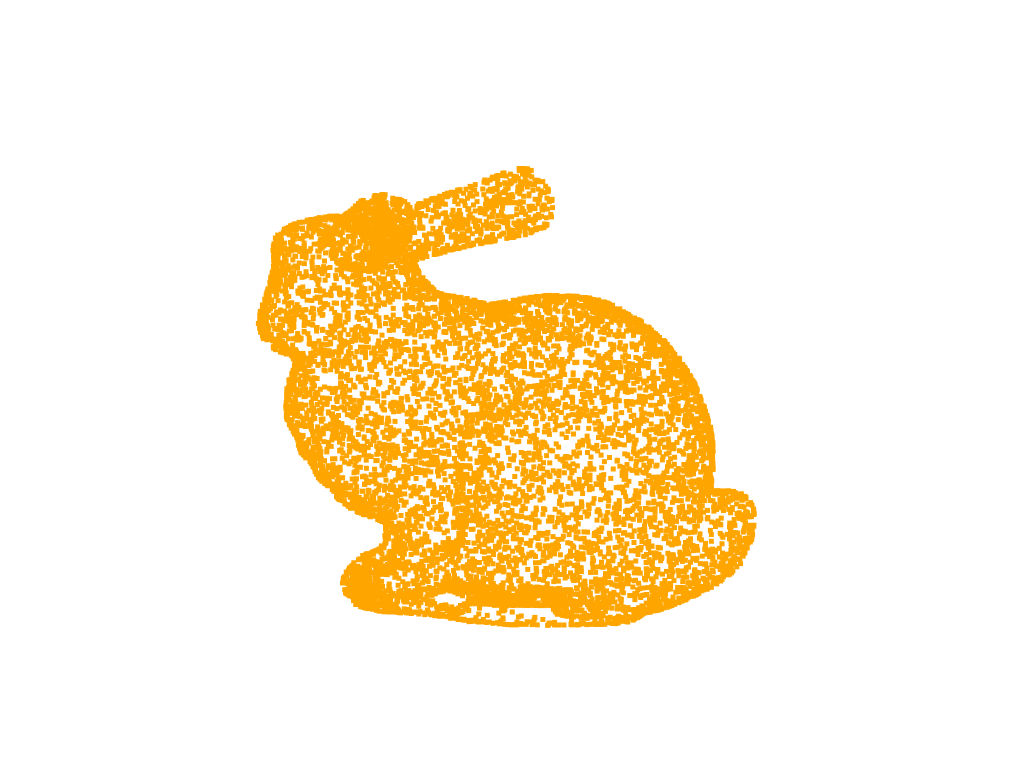}
    \includegraphics[width=0.25\linewidth]{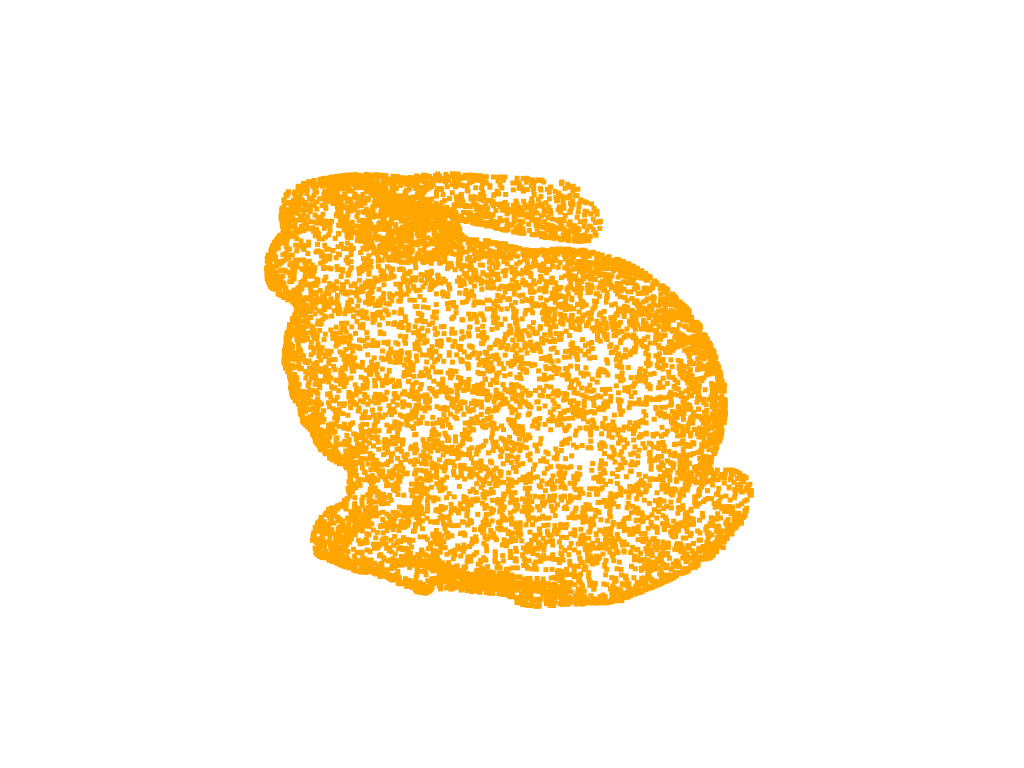} \\
    \includegraphics[width=0.25\linewidth]{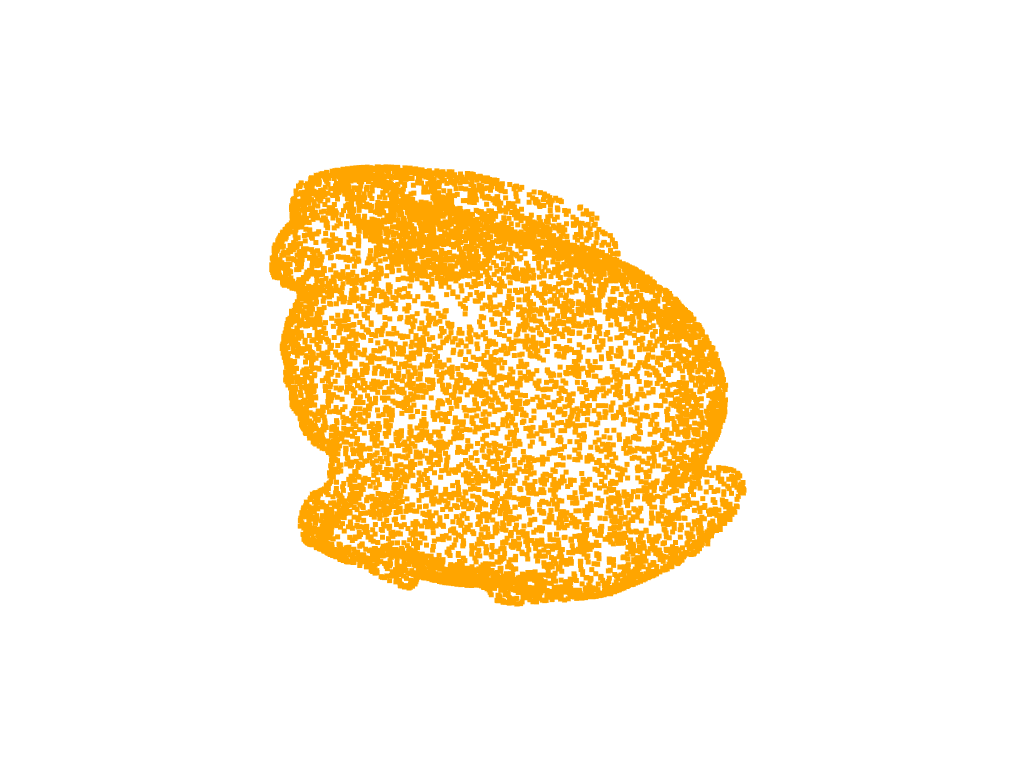}
    \includegraphics[width=0.25\linewidth]{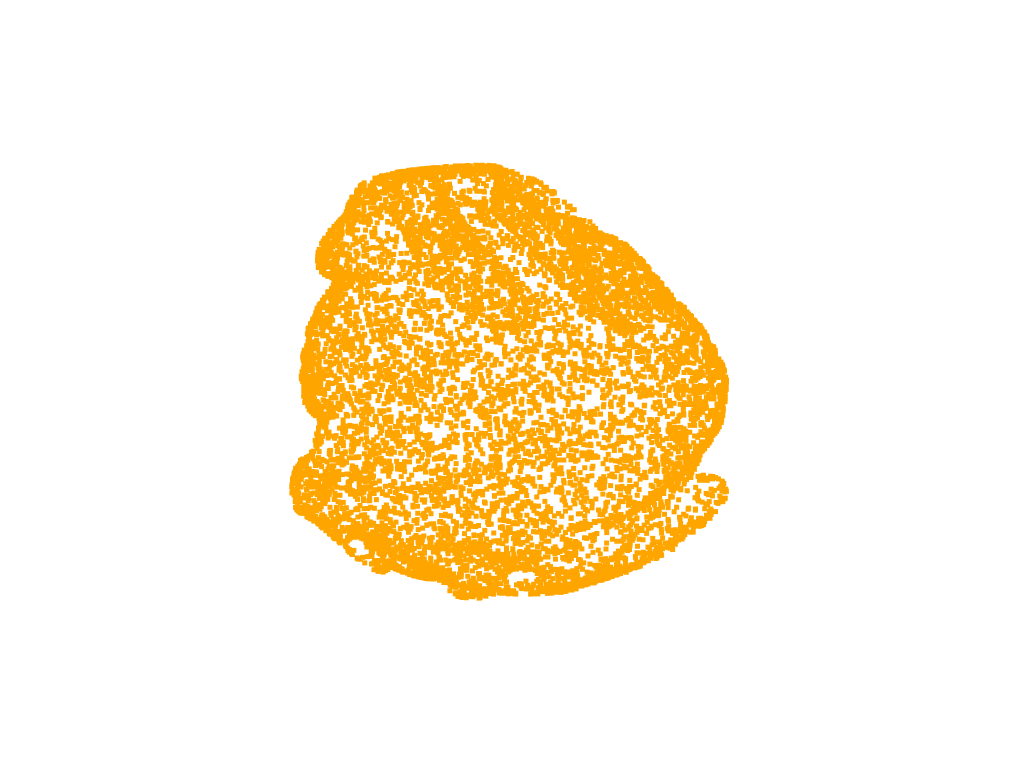}
    \includegraphics[width=0.25\linewidth]{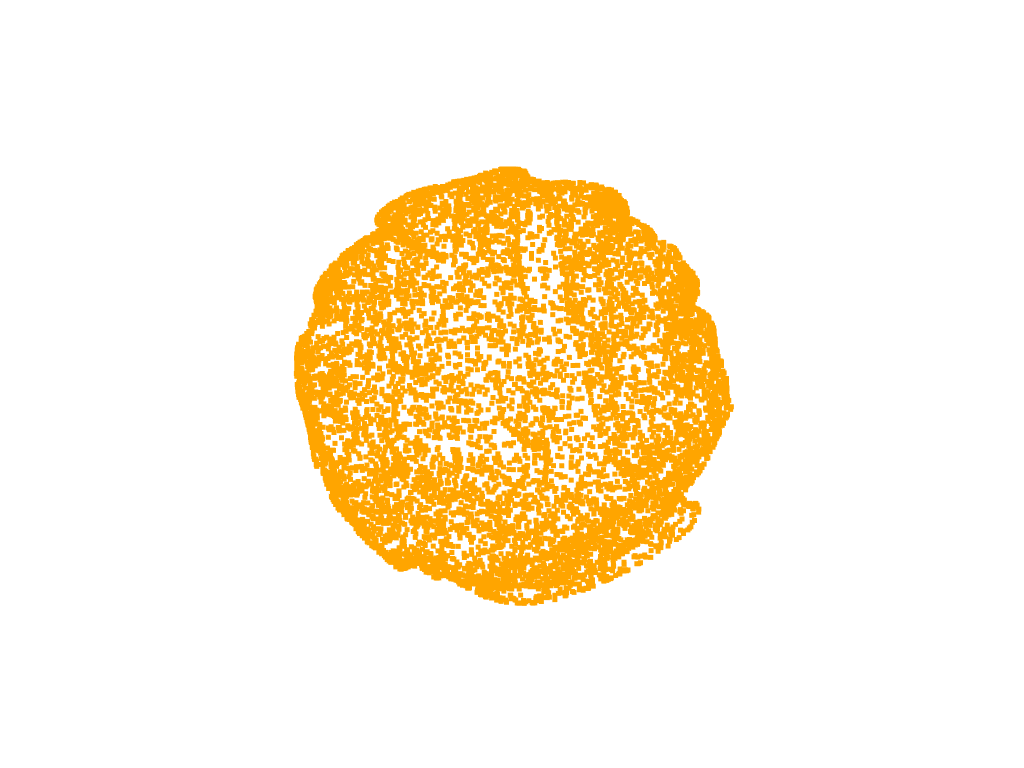}
    \end{tabular}
    \caption{Registration with a sliced ED kernel and a sliced ED loss on 10000 points.}
    \label{fig:3d_defo_bunny}
\end{figure}

%% file: appendix.tex
\section{Appendix}

\begin{lemma}
    Consider the convex set $B_{0,TV} \coloneqq \left\{ \Gamma \in \mathcal{M}_0(\Omega,\R^d) \,\bigg|\, |\Gamma|_{TV} \leq 1 \right\}$.
    Then, its set of extremal points  is
\begin{equation}\label{EqExtremalPoints}
        \operatorname{Ext}(B_{0,TV}) = \left\{ \sum_{i = 1}^k v_i \delta_{x_i}, x_i \neq x_j (i\neq j) \in \Omega \text{ s.t. } (v_1,\ldots,v_{k-1})  
  \text{ linearly independent and } \sum_{i = 1}^k | v_i| =1 \right\} \,.
    \end{equation}
\end{lemma}
\begin{proof}
    By the result of Fisher and Jerome \cite{FISHER197573}, the extreme points of the solutions to 
    $\min_{\mu} \| \mu \|_{TV}$ such that $\mu(\Omega) = 0$ are necessarily of the form $\sum_{i = 1}^k v_i \delta_{x_i}$ with $k\leq d$ and $v_i \neq 0$. Suppose now that $v_1,\ldots,v_{k-1}$ are not linearly independent. Then, there exists $\lambda_1,\ldots,\lambda_{k-1}$ such that $\sum_{i=1}^{k-1} \lambda_i v_i = 0$. Introduce $\delta = \sum_{i=1}^{k-1} \lambda_i v_i \delta_{x_i}$ and $\tilde \mu = \sum_{i=1}^{k-1} v_i$. Consider $\nu_1 = \tilde \mu - \varepsilon \delta + \frac{1}{2} v_n \delta_{x_n}$, $\nu_2 = \tilde \mu + \varepsilon \delta + \frac{1}{2} v_n \delta_{x_n}$ for $\varepsilon >0$. Then, by construction $\mu = \nu_1 + \nu_2$ for all $\varepsilon$ and since $\delta(\Omega) = 0$, $\nu_2(\Omega) = \nu_1(\Omega) = \mu(\Omega) =0$. For $\varepsilon$ sufficiently small, we also have $\| \nu_1 \|_{TV} = \sum_{i = 1}^{k-1} |1 - \varepsilon \lambda_i| |v_i| + \frac{1}2 |v_k|$ and $\| \nu_2 \|_{TV} = \sum_{i = 1}^{k-1} |1 + \varepsilon\lambda_i| |v_i| + \frac{1}2 |v_k|$. As a consequence, $\mu = \nu_1 + \nu_2$ and $\| \nu_1 \|_{TV} + \| \nu_2 \|_{TV} = \| \mu \|_{TV}$, which shows that $\mu$ cannot be an extremal point since it is a linear combination of $\frac{1}{\|\nu_1\|_{TV}}\nu_1$ and $\frac{1}{\|\nu_2\|_{TV}}\nu_2$.

    We now show that the measures in Formula \eqref{EqExtremalPoints} are indeed extremal. Assume that one of these measures $\mu$ can be written $\mu = t\nu_1 + (1-t)\nu_2$ where $t \in (0,1)$ and $\nu_1$, $\nu_2$ of unit TV norm. Consider the decomposition of $\nu_1$ (also $\nu_2$) w.r.t. $\mu$ which can be written as $\nu_1 = \tilde \nu_1 + \nu_1^\perp$ with $\nu_1^\perp$ singular w.r.t. $\mu$. Then, since $\mu = \tilde \nu_1 + \tilde \nu_2$ we get $\| \mu \| = 1$ implies $\nu_1^\perp = \nu_2^\perp = 0$.
    Now, we have $\nu_1 = \sum_{i = 1}^k u_i \delta_{x_i}$ and $\nu_2 = \sum_{i = 1}^k w_i \delta_{x_i}$ for some vectors $u_i,w_i$ such that $\sum_{i = 1}^k w_i = \sum_{i = 1}^k u_i = 0$ and $u_i + w_i = v_i$ for $i = 1,\ldots,k$. The TV norm of $\mu$ reads $1 = \sum_{i = 1}^k | v_i | \leq \sum_{i = 1}^k (| w_i | + |u_i|) = 1$ by the previous equalities. Therefore, by the Cauchy-Schwarz inequality, necessarily $w_i$ and $u_i$ are colinear and since $w_i + u_i = v_i$ they are colinear to $v_i$. Write $w_i = \alpha_i v_i$ and $u_i = \beta_i v_i$. Then, $\sum_{i = 1}^k \alpha_i v_i = 0$ and $\sum_{i = 1}^k \beta_i v_i = 0$. Since the rank of $(v_1,\ldots,v_k)$ is $k-1$, we get that $(\alpha_1,\ldots,\alpha_k)$ and $(\beta_1,\ldots,\beta_k)$ are necessarily colinear to $(1,\ldots,1)$. This implies that $\nu_1,\nu_2$ are colinear to $\mu$ and it leads to $\nu_1 = \nu_2 = \mu$.
\end{proof}